\documentclass[11pt, oneside]{amsart} 

\usepackage[english]{babel} 
\usepackage{enumerate}
\usepackage{graphics,color,pgf}
\usepackage{epsfig}
\usepackage[ansinew]{inputenc}
\usepackage[all,color,arrow]{xy}
\usepackage{tikz-cd}
\usepackage{bbm}
\usepackage{comment}
\usetikzlibrary{matrix,arrows,decorations.pathmorphing}
\newdir{ >}{!/8pt/@{}*@{>}}
\usepackage{amssymb, amsmath,amsthm, mathtools, amscd}
\usepackage{mathrsfs}
\usepackage[margin=1.5in]{geometry}
\usepackage{tikz}
\usetikzlibrary{automata, arrows.meta, positioning,arrows.meta,bending}
\usepackage[pagebackref, pdfborder={0 0 0}
  ,urlcolor=black,hypertexnames=false]{hyperref}

\usepackage[margin=1.5in]{geometry}
  \tikzcdset{scale cd/.style={every label/.append style={scale=#1},
    cells={nodes={scale=#1}}}}

\makeatletter
\newtheorem*{rep@theorem}{\rep@title}
\newcommand{\newreptheorem}[2]{%
\newenvironment{rep#1}[1]{%
 \def\rep@title{#2 \ref{##1}}%
 \begin{rep@theorem}}%
 {\end{rep@theorem}}}
\makeatother

\newtheorem{intro_thm}{Theorem}
\newtheorem{intro_cor}[intro_thm]{Corollary}

\newtheorem{lemma}{Lemma}[section]
\newtheorem{thm}[lemma]{Theorem} 
\newtheorem{prop}[lemma]{Proposition}
\newtheorem{cor}[lemma]{Corollary}

\theoremstyle{definition}
\newtheorem{defn}[lemma]{Definition}
\newtheorem{example}[lemma]{Example}
\newtheorem{con}[lemma]{Construction}
\newtheorem{es}[lemma]{Example}

\newtheorem{notation}[lemma]{Notation}
\newtheorem{no}[lemma]{\S}
\newtheorem{remark}[lemma]{Remark}

\theoremstyle{remark}

\newtheoremstyle{TheoremNum}
        {0.2 cm}{0.2 cm}              
        {\itshape}                      
        {}                              
        {}                     
        {.}                             
        { }                             
        {\thmname{\bfseries #1}\thmnote{ \bfseries #3}}
    \theoremstyle{TheoremNum}

\renewcommand{\L}{\mathfrak}

\newcommand\calM{{\mathcal M}}
\newcommand\calN{{\mathcal N}}
\newcommand\calG{{\mathcal G}}

\newcommand\calE{{\mathcal E}}

\newcommand\cF{{\mathcal F}}
\newcommand{\cE}{\mathcal E}
\newcommand{\cM}{\mathcal M}

\newcommand{\Hm}{\textup{H}}
\newcommand{\Hb}{\textup{H}_{\textup{b}}}
\newcommand{\Hmb}{\textup{H}_{\textup{mb}}}

\newcommand{\Hcb}{\textup{H}_{\textup{cb}}}

\newcommand{\Linf}{\mathrm{L}^{\infty}}

\newcommand{\rk}{\text{rk}}

\newcommand{\cG}{\mathcal{G}}

\newcommand{\cR}{\mathcal{R}}
\newcommand{\cH}{\mathcal{H}}
\newcommand{\cK}{\mathcal{K}}

\newcommand{\qand}{\quad \text{and} \quad}

\newcommand{\bR}{\mathbb{R}}
\renewcommand{\phi}{\varphi}

\newcommand{\cP}{\mathcal P}
\newcommand{\cB}{\mathcal B}
\newcommand{\cL}{\mathcal L}
\newcommand{\bN}{\mathbb N}

\newcommand{\ME}{\overset{\textup{ME}}{\sim}}

\begin{document}

\title[Bounded cohomological induction for transverse groupoids]{Bounded cohomological induction for transverse measured groupoids}

\author[T. Hartnick]{Tobias Hartnick}
\address{Institut f\"ur Algebra und Geometrie, KIT, Germany}
\email{tobias.hartnick@kit.edu}
\author[F. Sarti]{Filippo Sarti}
\address{Dipartimento di Matematica, Universit\'a di Pisa, Italy}
\email{filippo.sarti@dm.unipi.it}


\begin{abstract} 
We establish an induction isomorphism in the context of measurable bounded cohomology of discrete measured groupoid, which generalizes the Eckmann-Shapiro isomorphism in bounded cohomology of lattices due to Burger and Monod. 
In our wider setting, the role of lattices is taken by the class of transverse measured groupoids $(\cG, \nu)$ associated with a cross-section $Y$ in a pmp dynamical system $(X, \mu)$ of a lcsc group $G$ such that the associated hitting time process of $Y$ is locally integrable. Typical examples are given by pattern groupoids of strong approximate lattices. 

Under the assumptions that $G$ is unimodular we show that the measurable bounded cohomology of $(\cG, \nu)$ is isomorphic to the continuous bounded cohomology of $G$ with coefficients in $\Linf(X, \mu)$. As a consequence, if $G$ is amenable, then $(\cG, \nu)$ is boundedly acyclic, and in general the restriction map $\Hcb^\bullet (G; \bR) \to \Hmb^\bullet ((\cG, \nu);\underline{\bR})$ is injective. Moreover, it follows from known results in continuous bounded cohomology that if $G$ is a semisimple higher rank Lie group of Hermitian (respectively complex classical) type, then the second (respectively third) measurable bounded cohomology of $(\cG, \nu)$ is generated by the restriction of the bounded K\"ahler class (respectively bounded Borel class). These are the first explicit computations of non-trivial bounded cohomology groups of measured groupoids which are not isomorphic to an action groupoid.
\end{abstract}

\maketitle

\section{Introduction}

\subsection{Bounded cohomological induction and rigidity}

The interplay between locally compact groups and their lattices plays a central role in modern rigidity theory. Many properties of locally compact groups are inherited by their lattices, and many properties of discrete groups are inherited by their lattice envelopes, sometimes under additional cocompactness or integrability assumptions. For example, in geometric group theory the quasi-isometry type of a locally compact group is inherited by cocompact lattices, and many analytic properties of topological groups like amenability and Property (T) hold for a lattice in a locally compact group if and only if they hold for the ambient group. On the level of group cohomology (with trivial real coefficients, say), one can show that the continuous cohomology of a locally compact group injects into the cohomology of any of its \emph{uniform} lattices; this is a consequence of the classical Eckmann--Shapiro induction isomorphism in cohomology and does in general not hold for
non-uniform lattices. Similar, but stronger results hold also in bounded cohomology; since these results can be seen as the starting point of the present investigation, we will describe them in more detail.

Firstly, if $\Gamma$ is a lattice in a locally compact second countable (lcsc) group $G$, then by work of Burger and Monod \cite{BuMo} there is an isometric induction isomorphism (a la Eckmann-Shapiro)
\begin{equation}\label{InductionBCLattice}
\Hb^\bullet(\Gamma; \bR) \cong \Hcb^\bullet(G; \Linf (G/\Gamma)),
\end{equation}
where the left-hand side denotes bounded cohomology of $\Gamma$ with trivial real coefficients and the right-hand side denotes continuous bounded cohomology of $G$ with coefficients in the Banach $G$-module $\Linf (G/\Gamma)$. This implies, secondly, that the natural restriction map $\Hcb^\bullet(G; \bR) \to \Hb^\bullet(\Gamma; \bR)$ is injective. Finally, stronger results can be established for specific classes of groups. For example, it was showed by Monod \cite{MonodSarithmetic} that for semisimple Lie groups the restriction map happens to be an isomorphism within a certain range. In order to prove the induction isomorphism one makes use of the identification of continuous bounded cohomology with \emph{measurable} bounded cohomology, and indeed \eqref{InductionBCLattice} is better understood as a statement about \emph{measurable} bounded cohomology of locally compact groups.

Bounded cohomological induction is a central tool in establishing rigidity results for lattices in Lie groups; for example, it can be used to prove Mostow rigidity (via a dual version of the classical Gromov-Thurston argument \cite{BurgerBucherIozziRigidity}). Modern applications of bounded cohomological induction to rigidity theory started from the Burger--Monod proof of triviality of higher rank lattice actions on the circle \cite{BuMo} and include applications to K\"ahler rigidity \cite{BIKahler} and Hermitian higher Teichm\"uller theory \cite{BIWSurvey}. 

The present article is concerned with a generalization of Burger-Monod-Eckmann-Shapiro induction in the context of measurable bounded cohomology of discrete measured groupoid. In this wider context, the homogeneous dynamical system $G \curvearrowright G/\Gamma$ will be replaced by a general non-homogeneous pmp system, and the role of lattices will be taken by a class of discrete measured groupoids known as \emph{transverse measured groupoid} associated with certain cross sections $Y \subset X$.

\subsection{Measurable bounded cohomology of discrete measured groupoids}
In \cite{sarti:savini:23, sarti:savini25}, Alessio Savini and the second named author introduced a version of measurable bounded cohomology which applies to measured groupoids. We briefly recall the definition; see Section \ref{SecGroupoids} below for general background on measured groupoids. Given a discrete measured groupoid $(\cG, \nu)$ with target map $t$ we define a sequence of Borel spaces $\cG^{[n]}$ by
\[
\cG^{[n]} \coloneqq \{(g_0, \dots, g_{n-1}) \in (\cG^{(1)})^n \mid t(g_0) = \dots = t(g_{n-1})\}.
\]
This projects to $\cG^{(0)}$ with countable fibers via the map $(g_0, \dots, g_{n-1}) \mapsto t(g_0)$, and hence admits a unique $\sigma$-finite Borel measure which projects to $\nu$ and has fiber counting measures; this allows us to define a family of Banach spaces $\Linf(\cG^{[n]})$. 

We say that a function class $[f] \in \Linf(\cG^{[n+1]})$ is \emph{$\cG$-invariant} if for almost all $(g_0, \dots, g_n) \in \cG^{[n+1]}$ and all  $g \in \cG^{(1)}$ with $t(g_0) = t(g)$ we have
\[
f(g^{-1}g_0, \dots, g^{-1}g_n) = f(g_0, \dots, g_n).
\]
This defines a subspace  $\Linf(\cG^{[n]})^{\cG}  \subset \Linf(\cG^{[n]})$, and we have well-defined differentials
\begin{gather*}
d^n: \Linf(\cG^{[n+1]})^{\cG} \to \Linf(\cG^{[n+2]})^{\cG},\\
 d^nf(g_0, \dots, g_{n+1}) = \sum_{j=0}^{n+1}(-1)^j f(g_0, \dots, \widehat{g_j}, \dots, g_n).
\end{gather*}
Following \cite{sarti:savini:23} we then define the \emph{measurable bounded cohomology} of $(\cG, \nu)$ with coefficients in the trivial line bundle $\underline{\bR}$ over $\cG^{(0)}$ as
\[
\Hmb^\bullet((\cG,\nu); \underline{\bR}) \coloneqq \Hm^\bullet(0 \to  \Linf(\cG^{[1]})^{\cG} \to  \Linf(\cG^{[2]})^{\cG} \to  \Linf(\cG^{[3]})^{\cG} \to \dots).
\]

\subsection{Transverse measured groupoids}

Our first task is to identify a suitable replacement for lattices in the world of discrete measured groupoids; we take inspiration from the theory of approximate lattices \cite{BHApproximateLattices}, and more precisely the dynamical approach via transverse systems as considered in \cite{BHK, BHSiegel, ABC}.

Consider a probability-measure preserving (pmp) action $G \curvearrowright (X, \mu)$ of a unimodular lcsc group on a standard probability space and fix a Haar measure $m_G$ for $G$. We then obtain a Borel action groupoid $G \ltimes X$ which admits an invariant measure $\mu$ for the Haar system defined by $m_G$.

We say that a Borel subset $Y \subset X$ is a \emph{cross section} if $GY = X$ and for every $x \in GY$ the \emph{hitting time set} $Y_x \coloneqq \{g \in G \mid gx \in Y\}$ is non-empty and locally finite. In this case, the restriction Borel groupoid $\cG \coloneqq (G \ltimes X)|_Y$ has countable $t$-fibers, and the invariant probability measure $\mu$ on $X$ induces an invariant $\sigma$-finite Borel measure $\nu$ for $\cG$, variously known as the \emph{transverse measure} or \emph{Palm measure} of $(X, \mu, Y)$ (see \cite{ABC}). We are going to assume that $\nu$ is \emph{finite}, which amounts to local integrability of the point process $x \mapsto Y_x$; we then refer to $(X, \mu, Y)$ as an \emph{integrable transverse $G$-system} and to the discrete measured groupoid $(\cG, \nu)$ as the associated \emph{transverse measured groupoid}.

It turns out that transverse measured groupoids generalize lattices. Indeed, if $\Gamma$ is a lattice in a lcsc group $G$, then $G$ acts on $X \coloneqq G/\Gamma$ fixing a unique probability measure $\mu$ on $X$, and it we set $Y \coloneqq \{\Gamma\}$, then $(X, \mu, Y)$ is an integrable transverse $G$-system. The associated transverse measured groupoid has a single object $\{e\Gamma\}$ with automorphism group $\Gamma$ and can hence be identified with $\Gamma$; the transverse measure is the point measure at $\Gamma$ of total mass given by the covolume of $\Gamma$. By definition, the measurable bounded cohomology of $(\cG, \nu)$ is then just the bounded cohomology of $\Gamma$.

\subsection{Bounded cohomological induction for transverse measured groupoids}
From now on, let $G$ denote a unimodular lcsc group with Haar measure $m_G$. The main result of this article is the following version of Burger--Monod--Eckmann--Shapiro induction:
\begin{intro_thm}\label{Induction} If $(\cG, \nu)$ denotes the transverse measured groupoid of an integrable transverse $G$-system $(X, \mu, Y)$, then there is an isometric isomorphism
\[
\Hmb^\bullet((\cG,\nu); \underline{\bR}) \cong \Hcb^\bullet(G; \Linf (X, \mu)).
\]
\end{intro_thm}
\textbf{For the remainder of this subsection, $(\cG, \nu)$ denotes the transverse measured groupoid of an integrable transverse $G$-system $(X, \mu, Y)$.} We spell out a few basic consequences of Theorem \ref{Induction}; the most basic consequence is as follows: 
\begin{intro_cor}\label{intro_cor_amenable} 
  If $G$ is amenable, then $(\cG, \nu)$ is boundedly acyclic, i.e. $\Hmb^k((\cG,\nu); \underline{\bR})$ vanishes for all $k > 0$. 
\end{intro_cor}
To describe the consequences of Theorem \ref{Induction} beyond the amenable case, we introduce the following notation. Let $p$ be an admissible probability measure on $G$; then every class in $\Hcb^n(G; \bR)$ can be represented by
a  $G$-invariant continuous bounded $n$-cocycle $c: G^{n+1} \to \bR$, which is harmonic with respect to $p^{\otimes n}$. It turns out that
\[
\mathrm{res}(c): \cG^{[n+1]} \to \bR, \quad ((g_0, y_0), \dots, (g_n, y_n)) \mapsto c(g_0, \dots, g_n).
\]
defines a $\cG$-invariant bounded measurable $n$-cocycle for $\cG$ and that there is a well-defined restriction map
\[
\mathrm{res}^\bullet: \Hcb^\bullet(G; \bR) \to \Hmb^\bullet((\cG,\nu); \underline{\bR}), \quad [c] \mapsto [\mathrm{res}(c)].
\]
If we denote by $\Linf (X, \mu)_0 \subset \Linf (X, \mu)$ the subspace of functions of integral $0$, then we obtain:
\begin{intro_cor}\label{intro_cor_seminorms}
   The restriction map $\mathrm{res}^\bullet:  \Hcb^\bullet(G; \bR) \to \Hmb^\bullet((\cG,\nu); \underline{\bR})$ is injective with cokernel isomorphic to $\Hcb^\bullet(G; \Linf (X, \mu)_0)$ and does not increase seminorms.
\end{intro_cor}
If $G$ is a connected simple Lie group, then we can apply the results of Monod from \cite{MonodSarithmetic}. Indeed, the module $\Linf(X, \mu)_0$ is a semi-separable coefficient $G$-module, and if $\mu$ happens to be ergodic, then it does not admit any invariant vectors. By the aforementioned work of Monod, this implies:
 \begin{intro_cor}[Monod]\label{2rk} 
If $G$ is a connected simple Lie group and $\mu$ is $G$-ergodic, then for all $0\leq k< 2 \cdot \mathrm{rk}_{\mathbb{R}}(G)$ the restriction map $\mathrm{res}^k:  \Hcb^k(G; \bR) \to  \Hmb^k((\cG, \nu);\underline{\bR})$ is an isometric isomorphism.
 \end{intro_cor}
 In the setting of Corollary \ref{2rk}, it is conjectured that the continuous bounded cohomology $ \Hcb^\bullet(G; \bR)$ is isomorphic to the corresponding continuous cohomology; currently this is only known in low degrees. More precisely, by results of Burger and Monod \cite{BuMo}
 the space $\Hcb^2(G;\bR)$ is non-trivial if and only if $G$ is Hermitian, in which case it is one-dimensional, generated by the bounded K\"ahler class $\kappa_b^G$ \cite{BuMo}. Similarly, it was recently established in a series of papers  \cite{DLCM,HDLCM1,HDLCM2} that if $G$ is a complex simple Lie group of classical type, then $\Hcb^3(G; \bR)$ is one-dimensional and generated by the bounded Borel class $\beta_b^G$. If $G$ is Hermitian (respectively complex classical), we thus denote by $\kappa_b^{\cG} \coloneqq \mathrm{res}(\kappa_b^G)$ (respectively $\beta_b^{\cG} \coloneqq \mathrm{res}(\beta_b^{G}$)); then Corollary \ref{2rk} implies:
 \begin{intro_cor}[Burger-Monod, De la Cruz Mengual]\label{intro_cor_2BC}
   Let $(X,\mu,Y)$ be an ergodic transverse $G$-system where $G$ is a connected simple Lie group or real rank $\geq 2$.
 \begin{enumerate}[(i)]
 \item If $G$ is Hermitian, then $\Hmb^2((\cG,\nu); \underline{\bR}) = \bR \cdot \kappa_b^{\cG}$; otherwise $\Hmb^2((\cG,\nu); \underline{\bR}) = 0$.
 \item If $G$ is complex classical, then  $\Hmb^3((\cG,\nu); \underline{\bR}) = \bR \cdot \beta_b^{\cG}$.
 \end{enumerate}
\end{intro_cor}
These are the first computation of a non-trivial bounded cohomology groups of a discrete measured groupoid which is not isomorphic to an action groupoid. Let us mention in passing that
in view of recent results of Monod \cite{Flatmates} (in combination with \cite{MonodSarithmetic}) we also have the following vanishing result:
\begin{intro_cor}[Monod]\label{intro_cor_Archimedean}
   Let $(X,\mu,Y)$ be an ergodic transverse $G$-system where $G = \mathbf{G}(k)$ and $\mathbf{G}$ is a simple algebraic group over a non-Archimedean local field $k$. Then $ \Hmb^k((\cG, \nu);\underline{\bR}) = 0$ for all $0<k<2 \cdot \mathrm{rk}(G)$.
\end{intro_cor}

\subsection{Application to measured subsets} 

Our motivation in studying measurable bounded cohomology of transverse groupoids comes from the study of approximate lattices as introduced in \cite{BHApproximateLattices}; more generally, transverse measured groupoid arise naturally in the study of discrete subsets of locally compact groups. Let us briefly describe this connection.

Every lcsc group $G$ acts continuously on the space $\mathrm{Cl}(G)$ of closed subsets of $G$ via $g\cdot P \coloneqq Pg^{-1}$, where the latter is equipped with the Chabauty--Fell topology. Given a locally finite subset $P_o \subset G$ we denote by  $X^+(P_o)$ the orbit closure of $P_o$ in $\mathrm{Cl}(G)$ and define its \emph{hull} by $X(P_o) \coloneqq  X^+(P_o) \setminus \{\emptyset\}$. With the restricted action, this is a lcsc $G$-space which admits a canonical cross section given by
\[
Y(P_o) \coloneqq  \{Q \in X(P_o) \mid e \in Q\}.
\]
The associated transverse groupoid $\cG_{P_o}$ is called the \emph{pattern groupoid} of $P_o$ and contains combinatorial information about $P_o$. More precisely, we define the collection of \emph{marked patches} of $P_o$ as
\[
\mathcal P(P_o) \coloneqq \{(B \cap P_og^{-1}, B) \mid B \subset G \text{ bounded}, g \in G\}.
\]
If $(\alpha, B) \in \mathcal P(P_o)$, then $\alpha$ is called a \emph{$B$-patch} of $P_o$, and we say that $P_o$ has \emph{finite local complexity} (FLC) if it has finitely many $B$-patches for every bounded $B\subset G$; equivalently, the symmetrization $S(P_o) \coloneqq P_oP_o^{-1}$ does not accumulate at the identity. If $P_o$ has FLC, then the associated pattern groupoid is given by
\[
\cG_{P_o}^{(0)} = \{Q \subset S(P_o) \mid e \in Q, \cP(Q) \subset \cP(P_o)\} \qand \cG_{P_o}^{(1)} = \{(q,Q) \mid q \in Q \in \cG_{P_o}^{(0)}\}.
\]
If $P_o$ and $P_1$ are FLC sets in possibly different lcsc groups $G$ and $H$, then they are called \emph{pattern-isomorphic} if the associated pattern groupoids are isomorphic. This is the case if the local groups $S(P_o)$ and $S(P_1)$ (with partial group operations given by restriction from the ambient lcsc group) are isomorphic and implies that there there is bijection between $\cP(P_0)$ and $\cP(P_1)$ which is compatible with inclusions, i.e.\ $P_o$ and $P_1$ are combinatorially equivalent in the sense that they have ``matching patches''. The \emph{pattern isomorphism problem} for FLC sets asks for a classification of FLC sets in lcsc groups up to pattern-isomorphism or, more realistically, for isomorphism invariants of pattern groupoids a.k.a.\ \emph{pattern invariants}.

We say that a FLC set $P_o \subset G$ is \emph{measured} if the hull system $G \curvearrowright X(P_o)$ admits an ergodic invariant probability measure $\mu \in \mathrm{Prob}_{\mathrm{erg}}(X(P_o))^G$. Examples of measured FLC sets include strong approximate lattices in the sense of \cite{BHApproximateLattices}, and in particular regular model sets in the sense of \cite{BHP1}. One can show that in the case of a measured FLC set the transverse system $(X(P_o), \mu, Y(P_o))$ is automatically integrable, hence gives rise to a finite invariant measure $\nu$ on $\cG_{P_o}$. We define the \emph{bounded cohomology} of $P_o$ with respect to $\mu$ as
\[
\Hb^\bullet((P_o, \mu); \bR) \coloneqq \Hmb^\bullet((\cG_{P_o}, \nu); \underline{\bR}).
\]
If we assume that $P_o$ has \emph{finite Voronoi volume} in the sense that all Voronoi cells of $P_o$ in $G$ have uniformly bounded Haar volume (e.g.\ if $P_o$ is cocompact), then one can show that every finite invariant measure on $\cG_{P_o}$ is proportional to a transverse measure of an invariant probability measure on $\mu$. This then implies that for every $k \geq 0$ the set 
\[
\Hb^k(P_o; \bR) \coloneqq \{[\Hb^k((P_o, \mu); \bR)] \mid \mu \in \mathrm{Prob}_{\mathrm{erg}}(X(P_o))^G\},
\]
where $[\cdot]$ denotes equivalence classes up to isometric isomorphism, is a pattern invariant of $P_o$. Theorem \ref{Induction} now allows us to compute this pattern invariant in certain cases:
\begin{intro_cor}\label{ApproxLatt1}
If $P_o \subset G$ is a measured FLC set of finite Voronoi volume, then
\[
\Hb^k(P_o; \bR) = \{\Hcb^k(G; \Linf(X, \mu)) \mid  \mu \in \mathrm{Prob}_{\mathrm{erg}}(X(P_o))^G\}.
\]
\end{intro_cor}
\begin{intro_cor}\label{ApproxLatt2} If $G$ is a higher rank simple Lie group and $P_o \subset G$ is a strong approximate lattice of finite Voronoi volume, then
\[
\Hb^2(P_o; \bR) = \left\{\begin{array}{ll}\{[\bR]\},& \text{if $G$ is of Hermitian type};\\ \{[0]\}, &\text{else.}\end{array} \right.
\]
\end{intro_cor}
This implies that a strong approximate lattice in a higher rank simple Lie group of Hermitian type cannot be pattern isomorphic to an approximate lattice in a higher rank Lie group which is not of Hermitian type (assuming both are of finite Voronoi volume). While in principle this indicates the usefulness of bounded cohomology in classifying approximate lattices up to pattern isomorphism, we should point out that (as observed by Bj\"orklund and the first named author) in the context of higher rank Lie groups much stronger pattern invariants arise via Furman's theory of measure equivalence rigidity \cite{Furman99}.

\subsection{An application to discretization}

Let $G$ be a unimodular lcsc group and denote by $G^\delta$ the underlying discrete group. If $G$ admits a lattice $\Gamma$, then the corresponding restriction map factors as
\[
\Hcb^\bullet(G; \bR) \to \Hb^\bullet(G^\delta; \bR) \to\Hb^\bullet(\Gamma; \bR).
\]
Since the restriction map is injective, we can thus infer from the existence of a lattice in $G$ the injectivity of the \emph{discretization map} $\Hcb^\bullet(G; \bR) \to \Hb^\bullet(G^\delta; \bR)$. A similar argument can now also be run in the groupoid context. If we define the \emph{return time set} of a transverse $G$-system $(X, \mu, Y)$ as $\Lambda(Y) \coloneqq \{g \in G \mid gY \cap Y \neq \emptyset\}$, then this leads to the following result:
\begin{intro_cor}\label{Discretization} If there exists an integrable transverse $G$-system with discrete return time set, then the discretization map $\Hcb^\bullet(G; \bR) \to \Hb^\bullet(G^\delta; \bR)$ is injective.
\end{intro_cor}
It is well-known that \emph{every} lcsc group admits an integrable transverse system, but it seems to be an open problem whether we can always arrange the return time set to be discrete. This is related to the question whether every lcsc group admits a measured subset $P_o$ such that $P_oP_o^{-1}$ is locally finite.

\subsection{On the proof of Theorem \ref{Induction}}
The first ingredient in our proof of Theorem \ref{Induction} is the simple observation (due to Bj\"orklund and the first named author) that transverse groupoids are \emph{measure equivalent} to the ambient lcsc group in a suitable sense. This can be seen as a measurable version of the classical fact that $G \ltimes X$ and $\cG$ are Morita-equivalent if the defining cross section $Y$ happens to be cocompact.

Our second ingredient is a result of Monod and Shalom \cite{MonShal} which states that bounded cohomology of discrete groups is well-behaved under measure equivalence, or rather its straight forward generalization to discrete measured groupoids. Combining both ingredients we deduce that Theorem \ref{Induction} holds if $G$ happens to be discrete and $\mu$ is ergodic.

Going beyond the discrete case, we expect that measurable bounded cohomology of (non-discrete) measured Borel groupoids is also well-behaved under suitably-defined measure equivalence, which would imply Theorem \ref{Induction} (at least for ergodic actions). However, we are currently lacking the technology concerning Borel groupoids to establish such a general result. 

In view of this shortfall, our proof in the non-discrete case works by imitating (rather than generalizing) the proof in the discrete case and making a number of (rather non-trivial) technical adjustments by hand. This is where our third ingredient comes in, a construction of a cohomological transfer spaces, which is inspired by a construction of Ph.\ Blanc \cite{Blanc} in the lattice case, or rather a variant of this construction as presented by Monod in \cite{monod:libro}. As a by-product, we also see that no ergodicity hypothesis is required in Theorem \ref{Induction}.

\subsection{Organization of the article} Besides this introduction, this article contains $6$ further sections. In Section \ref{SecGroupoids} we recall basic notions about measured groupoids (Sections \ref{sec2.1}, \ref{sec2.2} and \ref{sec2.3}), measurable bundles (Sections \ref{sec2.4}) and bounded cohomology of measured groupoids (Section \ref{sec2.5}).
In Section \ref{sec:trans:system} we discuss transverse systems (Section \ref{sec3.1}) and transverse measured groupoids (Section \ref{sec:3.2}); we then explain how actions of a transverse groupoid on Lebesgue spaces can be induced by actions of the ambient groups and discuss the associated bundles (Section \ref{sec:3.3}). In Section \ref{sec:BCandME} we establish a version of our main theorem in the special case where $G$ is discrete, using the theory of measure equivalence; we then point out the technical reasons which currently prevent us from extending this argument to the non-discrete case. We introduce couplings for discrete measured groupoids (Section \ref{sec:4.1}), we show how bounded cohomology behaves under couplings (Section \ref{sec:4.2}) and we apply this machinery to transverse groupoids (Section \ref{sec:4.3}). 
We then turn to the proof of our main theorem in the non-discrete case. Section \ref{sec:trans:isom} introduces our main technical tools, namely transfer spaces (Section \ref{sec:trans:space}), transfer measures (Section \ref{sec:trans:meas}) and the transfer isomorphism (Section \ref{sec:isomo}), and based on these tools the actual proof is given in Section \ref{sec:amenable}. A key is step is to show that a Lebesgue space for the transverse groupoid which is induced from an amenable Lebesgue space of the ambient lcsc group is itself amenable. We prove this in Section \ref{sec:6.1} and use it to construct resolutions computing the measurable bounded cohomology of transverse measured groupoids in Section \ref{SubsecAmenable}.
We then conclude the proof of the main theorem in Section \ref{sec:6.2}. We conclude in Section \ref{SecPoisson} with an explicit implementation of the induction isomorphism on the cocycle level. In Section \ref{sec:res:ind} we focus on some applications. We introduce restriction and induction maps (Section \ref{sec:7.1}) and then establish the various corollaries mentioned in the introduction (Section \ref{sec:corollaries}).  

\subsection*{Acknowledgements} The second named author was partially supported by INdAM through GNSAGA and by MUR through the PRIN project ``Geometry and topology of manifolds".

\section{Background on measurable bounded cohomology of groupoids}\label{SecGroupoids}

\subsection{Groupoids and their actions}\label{sec2.1}

\begin{no}\label{BasicGroupoid}  A \emph{groupoid} is a small category $\mathcal{G}$ whose morphisms are all invertible. We denote the set of objects and the set of morphisms respectively by $\mathcal{G}^{(0)}$ and $\mathcal{G}^{(1)}$ and, given a composable pair $(g,h) \in \cG^{(1)}\times \mathcal{G}^{(1)}$, we write $gh \coloneqq  g \circ h$. The \emph{target} and the \emph{source} maps are indicated respectively by $t:\mathcal{G}^{(1)}\to \mathcal{G}^{(0)}$ and by $s:\mathcal{G}^{(1)}\to \mathcal{G}^{(0)}$. Given a groupoid $\cG$ and $x,y \in \cG^{(0)}$ we define subsets of $\cG^{(1)}$ by
\[\cG_x \coloneqq  s^{-1}(x), \quad \cG^y \coloneqq  t^{-1}(y) \qand \cG_x^y \coloneqq  \cG_x \cap \cG^y.\] 
We often consider $\cG^{(0)}$ as a subset of $\cG^{(1)}$ via the embedding $x \mapsto 1_x$. Under this identification we then have $s(g) = g^{-1}g$ and $t(g) = gg^{-1}$.
The \emph{underlying equivalence relation} $R_{\cG}$ of $\cG$ is the equivalence relation on $\cG^{(0)}$ given by $xR_{\cG} y$ iff $\cG_x^y \neq \emptyset$. We say that a subset $U\subset \mathcal{G}^{(0)}$ is \emph{invariant} if it is invariant under $R_{\cG}$, i.e.\ \[
\forall\, g \in \cG^{(1)}: \; s(g) \in U \iff t(g) \in U.
\]
A groupoid is called \emph{principal} if $|\cG_x^y| \leq 1$ for all $x,y \in \cG^{(0)}$; in this case we have  $|\cG_x^y| = 1$ if and only if $x R_{\cG} y$, and hence $\cG$ is uniquely determined by $R_{\cG}$. 
\end{no}
 \begin{notation} Given sets $A, B, C$ and maps $f: A \to C$ and $g: B \to C$ we denote by
\[
A \; {}_f\negthickspace \times_g B \coloneqq  \{(a,b) \in A \times B \mid f(a) = g(b)\}
\]
the associated \emph{fiber product}; if $f$ and $g$ are clear from context we also write $A \times_C B$. In particular, if $\cG$ is a groupoid, $A$ is a set and $t_A: A \to \cG^{(0)}$ is a map, then we write
\[
 \mathcal{G}^{(1)} \times_{\mathcal{G}^{(0)}}A \coloneqq  \mathcal{G}^{(1)}\; {}_s \negthickspace \times_{t_A} A.
\]
We denote by $\cG^{(2)} \coloneqq \cG^{(1)} {}_s \negthickspace \times _t \cG^{(1)}$ the space of composable arrows and by $\cG^{(n)} \coloneqq \cG^{(1)}  {}_s \negthickspace \times _t \dots \ {}_s \negthickspace \times_t \cG^{(1)}$ the space of paths of length $n$ in $\cG$.
\end{notation}
\begin{no} Let $\cG$ be a groupoid. A \emph{left $\mathcal{G}$-space} is a set $A$ together with an \emph{anchor map} $t_A:A\to \cG^{(0)}$ and an \emph{action map} $\cG^{(1)} \times_{\mathcal{G}^{(0)}} A \to A$, $(g,a) \mapsto g\cdot a$ such that the following properties hold:
\begin{enumerate}[(i)]
\item $t_A(g\cdot a)=t(g)$ for all  $(g,a) \in \mathcal{G}^{(1)}\times_{\mathcal{G}^{(0)}} A$;
\item $1_y\cdot a = a$ for all $y \in \cG^{(0)}$ and $a \in t_A^{-1}(y)$;
\item  $(gh)\cdot a=g\cdot (h\cdot a)$, whenever $(gh,a)\in \mathcal{G}^{(1)}\times_{\mathcal{G}^{(0)}}A$ and $(g,h\cdot a) \in \mathcal{G}^{(1)} \times_{\mathcal{G}^{(0)}} A $.
\end{enumerate}
We also say that \emph{$\mathcal{G}$ acts on $A$} and write $\mathcal{G}\curvearrowright A$. Note that (ii) and (iii) imply in particular that $g^{-1}\cdot (g\cdot a) = a$ for all $(g,a) \in  \mathcal{G}^{(1)}\times_{\mathcal{G}^{(0)}} A$.
\end{no}
\begin{example}\label{ActionExamples} The following are examples of groupoid actions:
\begin{enumerate}[(i)]
\item Every group $G$ can be considered as a groupoid $\underline{G}$ with a single object $*$ with automorphism group $G$; then a $\underline{G}$-action is just a $G$-action in the usual sense.
\item  Every groupoid $\cG$ acts on $\cG^{(0)}$ with anchor map the identity and action given by $g x = t(g)$ if $s(g) = x$, and on 
$\cG^{(1)}$ with anchor map $t$ and action given by $g_1 \cdot g_2 \coloneqq  g_1 g_2$.
\item If a groupoid acts on a set $A$ with anchor map $t_A: A \to \cG^{(0)}$, then it also acts on the fiber product $A^{[n]}\coloneqq A {}_{t_A}\negthickspace\times \ldots\negthickspace \times_{t_A} A$: The anchor map is given by 
\[
t_A(a_0, \dots, a_{n-1}) = t_A(a_0)  = \dots = t_A(a_{n-1}),
\]
and if $t_A(a_0, \dots, a_{n-1})=  s(g)$, then
 $g\cdot (a_0, \dots, a_{n-1}) = (g\cdot a_0, \dots, g\cdot a_{n-1})$. We refer to this as the \emph{diagonal action} of $\cG$ on $A^{[n]}$. 
 \item Combining (ii) and (iii), every groupoid $\cG$ acts diagonally on \[\cG^{[n]} =  \{(g_0, \dots, g_{n-1}) \in \cG^{(1)} \mid t(g_0) = \dots = t(g_{n-1})\}\] by left-multiplication with anchor map $t(g_0,\dots, g_{n-1})\coloneqq t(g_0) = \dots = t(g_{n-1})$.
\end{enumerate}
\end{example}
\begin{con}[Action groupoid]
If $A$ is a left $\cG$-space, then the associated \emph{left action groupoid} $\mathcal{L} = \cG \ltimes A$ is given by $\mathcal{L}^{(0)} = A$, $\mathcal{L}^{(1)} = \cG^{(1)} {}_A \times_{t_A} A$ with structure maps $s(g,a) = a$, $t(g,a) =g\cdot a$ and
$(g, a)(h,b) = (gh,b)$ if $a = h\cdot b$. 

Similarly, the \emph{right action groupoid} $\mathcal{R} = A \rtimes \cG$ is given by $\mathcal{R}^{(0)} = A$, $\mathcal{R}^{(1)} = A \; {}_{t_A}\negthickspace \times_t \cG^{(1)}$ with structure maps $s(a,g)=g^{-1}\cdot a$, $t(a,g)=a$ and $(a,g)(b,h) = (a,gh) $ if $g^{-1}\cdot a = b$. 

In particular, given a group action $G \curvearrowright A$ we obtain groupoids $G \ltimes A$ and $A \rtimes G$.
 \end{con}
\begin{no}
Using action groupoids we can efficiently define morphisms between $\cG$-sets: If a groupoid $\cG$ acts on sets $X$ and $Y$, then a map $\pi: X \to Y$ is called \emph{$\cG$-equivariant} if there is a functor $\pi_*: \cG \ltimes X \to \cG \ltimes Y$ satisfying
\[
\pi^{(0)}_*(x) = \pi(x) \qand \pi^{(1)}_*(\gamma, x) = (\gamma, \pi(x)).
\]
We can also use action groupoids to define invariant subsets. By definition, a subset $U$ of a $\cG$-set $A$ is \emph{$\cG$-invariant} if it is $(\cG \ltimes U)$-invariant in the sense of \S\ref{BasicGroupoid}. More explicitly, this means that $U = \cG U$, where the latter is defined as
\[
\cG U = \{g \cdot u \mid (g, u) \in \cG^{(1)} \times_{\cG^{(0)}}A\}.
\]
\end{no}
\begin{con}[Restriction] If $\cG$ is a groupoid and $Y \subset \cG^{(0)}$ is a subset, then the  \emph{restriction} $\cG|_Y$ is the groupoid with sets of objects and morphisms given by
\[
\cG|_Y^{(0)} \coloneqq Y \qand \cG|_Y^{(1)} \coloneqq \{g \in \cG^{(1)} \mid \{s(g), t(g)\} \subset Y\}
\]
and structure maps given by the restrictions of the structure maps of $\cG$.
\end{con}

\begin{no} If $A$ is a left $\cG$-space, then we denote by $R_A$ the underlying equivalence relation of $R_{\cG \ltimes A}$. The associated partition of $X$ is then denoted by $\cG\backslash A$, and we denote by $\pi: A \to \cG\backslash A$, $a \mapsto [a]$ the canonical projection. In the case of a group action, this recovers the usual quotient of $A$ by $\cG$. We say that the $\cG$-action is \emph{free}, if $\cG \ltimes A$ (or, equivalently, $A \rtimes \cG$) is a principal groupoid. In this case, 
a subset $\cF \subset A$ is called a \emph{strict fundamental domain} for the $\cG$-action if it intersects each class in $R_{\cG \ltimes A}$ in precisely one element; then $\pi$ restricts to a bijection $\cF \to \cG \backslash A$.
 \end{no}
 
 \subsection{Borel groupoids and Haar systems}\label{sec2.2}
\begin{no}\label{borel groupoids} A measurable space is called a \emph{standard Borel space} if it is isomorphic to a Borel subset of a Polish space. If $(A, \cB)$ is a standard Borel space and $\tau$ is a $\sigma$-finite Borel measure on $\cB$, then $(A, \cB, \tau)$ is called a \emph{Lebesgue space}; if moreover $\tau(A) = 1$ then it is called a \emph{standard probability space}. For sake of brevity, we will usually omit the $\sigma$-algebra $\mathcal{B}$ from the notation.

Given a Lebesgue space $(A, \tau)$ we denote by $\cL^\infty(A)$ the space of bounded real-valued Borel functions on $A$ and by $\Linf(A,\tau)$ the space of equivalence classes of such functions, where two functions are considered equivalent if they agree $\tau$-almost everywhere. If the measure $\tau$ is clear from context we also write $\Linf(A) \coloneqq \Linf(A, \tau)$.
\end{no}
\begin{no} Let $\pi:A\to B$ be a Borel map between standard Borel spaces. A \emph{Borel system} of measures for $\pi$ is a collection $\lambda =\{\lambda^b\}_{b\in B}$ of $\sigma$-finite measures with $\lambda^b(A\setminus \pi^{-1}(b))=0$ for every $b\in B$ and such that the map 
\[B\to \mathbb{R}\,,\;\;\;b\mapsto \lambda^b(f)\] is Borel for every Borel function $f:A\to \mathbb{R}$. 
\end{no}
\begin{no} Let  $\mathcal{G}$ be a groupoid and let $\cB$ be a $\sigma$-algebra of subsets of $\cG^{(1)}$; we equip $\cG^{(0)} \subset \cG^{(1)}$ and $\cG^{(2)} \subset \cG^{(1)} \times \cG^{(1)}$ with the respective trace $\sigma$-algebras. We then say that $(\cG, \cB)$ (or just $\cG$ for short) is a \emph{Borel groupoid} if $\mathcal{G}^{(0)}$ and $\mathcal{G}^{(1)}$ are standard Borel spaces and multiplication $\cG^{(2)} \to \cG^{(1)}$ and inversion $\cG^{(1)} \to \cG^{(1)}$ are Borel; by \S \ref{BasicGroupoid} this implies that also the source and target map are Borel, and hence $\cG^{(2)}$ is standard Borel.
\end{no}
\begin{example}\label{ResBorel}
(i) If $\cG$ is a Borel groupoid and $A$ is a standard Borel space, then $A$ together with an action of $\cG$ is called a \emph{Borel $\cG$-space} if the anchor map and the action map are Borel; this implies that $\cG \ltimes A$ and $A \rtimes \cG$ are again Borel groupoids. In particular, if $G$ is a lcsc group and $A$ is a Borel $G$-space, then $G \ltimes A$ and $A \rtimes G$ are Borel groupoids.

\item (ii) If $\cG$ is a Borel groupoid and $Y \subset \cG^{(0)}$ is Borel, then the restriction $\cG|_Y$ is a Borel groupoid when equipped with the trace $\sigma$-algebra.
\end{example}
\begin{defn} If $\mathcal{G}$ is a Borel groupoid, then a Borel system $\lambda = \{\lambda^y\}_{y\in \mathcal{G}^{(0)}}$ for $t:\mathcal{G}^{(1)}\to \mathcal{G}^{(0)}$ of measures is called a \emph{Haar system} provided
\[
  \int_{\mathcal{G}^{(1)}}f(gh) d\lambda^{s(g)}(h)=\int_{\mathcal{G}^{(1)}}f(h)d\lambda^{t(g)}(h) \quad \text{for every }g \in \cG^{(1)}.
\]
\end{defn}
\begin{remark} The existence of a Haar system is a restrictive condition for a Borel groupoid; it forces the groupoid to be the underlying Borel groupoid of a lcsc groupoid with open source and target maps. Moreover, Haar systems are not unique in any reasonable sense. 
\end{remark}
\begin{example} The following examples provide Haar systems in all cases which are of interest to us here:
\item (i) If $G$ is a lcsc group, then a Haar system for $\underline{G}$ is the same a left-Haar measure for $G$.
\item (ii) If $\cG$ is a groupoid whose target map has countable fibers, then a Haar system for $\cG$ is given by counting measures on fibers; this is called the \emph{counting Haar system}.
\item (iii) Let $\cG$ be a groupoid with Haar system $\lambda$ and let $A$ be a Borel $\cG$-space with anchor map $t_A: A \to \cG^{(0)}$. Then a Haar system for $\cR \coloneqq A \rtimes \cG$ is given by $\lambda_A = (\delta_a \otimes \lambda^{t_A(a)})_{a \in A}$. Indeed, if $(a,g) \in \cR^{(1)}$, then $t_A(a) = t(g)$ and $t_A(g^{-1} \cdot a) = s(g)$, hence
\begin{align*}
&\int_{\cR^{(1)}} f((a,g)(b,h)) \, d\lambda_A^{s(a,g)}(b,h)  = \int_{\cG^{(1)}} f(a, gh) d\lambda^{t_A(g^{-1} \cdot a)}(h) \\
=& \int_{\cG^{(1)}} f(a, gh)\, d\lambda^{s(g)}(h)
 =  \int_{\cG^{(1)}} f(a,h)  d\lambda^{t(g)}(h) 
 =  \int_{\cG^{(1)}} f(a,h)  d\lambda^{t_A(a)}(h)\\
  =& \int_{\cR^{(1)}} f(b,h) \, d\lambda_A^{t(a,g)}(b,h).
\end{align*}
\item (iv) In the situation of (iii) we can also define a Haar system $\lambda'_A$ for $\cG \ltimes A$ as the push forward of $\lambda_A$ via the Borel isomorphism 
\[
(A \rtimes \cG)^{(1)} \to (\cG \ltimes A)^{(1)}, \quad (a,g) \mapsto (g^{-1}, g^{-1}\cdot a).
\]
\end{example}

\subsection{Measured groupoids and their Lebesgue spaces}\label{sec2.3}
\begin{no} Let $(\cG,\lambda)$ be a Borel groupoid with Haar system.

\item (i) Given a $\sigma$-finite Borel measure $\nu$ on $\cG^{(0)}$ we define a $\sigma$-finite Borel measure $\nu \circ \lambda$ on $\cG^{(1)}$ by
  \[\nu\circ\lambda(f)=\int_{\mathcal{G}^{(0)}} \int_{\mathcal{G}^{(1)}} f(g)d\lambda^y(g)d\nu(y).\]
We say that $\nu$ is \emph{invariant} (respectively, \emph{quasi-invariant}) under $(\cG, \lambda)$ if $\nu\circ\lambda$ (respectively its measure class) is invariant under the inversion map $g \mapsto g^{-1}$. If $\lambda$ is clear from context, we simply say that $\nu$ is invariant (or quasi-invariant) under $\cG$.
  
  \item (ii) More generally, let $A$ be a Borel $\cG$-space and let $\tau$ be a $\sigma$-finite Borel measure on $A$. We then  say that $\tau$ is \emph{invariant} (respectively, \emph{quasi-invariant}) under $(\cG, \lambda)$ if it is invariant (respectively quasi-invariant) under $(A \rtimes \cG, \lambda_A)$. 
  
  More explicitly, this means that the measure $\tau \circ \lambda_A$ on $(A \rtimes \cG)^{(1)}$ as given by 
  \[
 ( \tau \circ \lambda_A)(f) = \int_A \int_{\cG^{(1)}} f(a,g) \, d \lambda^{t_A(a)}(g) d \tau(a)
  \]
  is invariant (respectively quasi-invariant) under the map $(a,g) \mapsto (g^{-1} \cdot a, g^{-1})$.
\end{no}
\begin{defn} Let $(\cG, \lambda)$ be a Borel groupoid with Haar system.
\begin{enumerate}[(i)]
\item If $\nu$ is a $\cG$-quasi-invariant $\sigma$-finite Borel measure on $\cG^{(0)}$, then $(\cG, \nu, \lambda)$ is called a \emph{measured groupoid}.
\item If $A$ is a Borel $\cG$-space and $\tau$ is a $\cG$-quasi-invariant $\sigma$-finite Borel measure on $A$, then $(A, \tau)$ is called a \emph{Lebesgue $(\cG, \lambda)$-space}.
\item If $(\cG, \nu, \lambda)$ is a measured groupoid then a \emph{Lebesgue $(\cG, \nu, \lambda)$-space} $(A, \tau)$ is a Lebesgue $(\cG, \lambda)$-space such that $(t_A)_*\tau = \nu$.
\end{enumerate}
\end{defn}
\begin{remark} If $(A, \tau)$ is a Lebesgue $(\cG, \nu, \lambda)$-space, then there exists a Borel system $\{ \tau^y\}_{y \in \cG^{(0)}}$ for $t_A: A \to \cG^{(0)}$ such that
\begin{equation}\label{taudisintegration}
\tau(f) = \int_{\cG^{(0)}} \int_A f(a) \, d\tau^y(a) \, d \nu(y).
\end{equation}
\end{remark}
\begin{defn}
A measured groupoid $(\mathcal{G},\nu,\lambda)$ is called \emph{ergodic} if every invariant Borel subset $U\subset \mathcal{G}^{(0)}$ satisfies either $\nu(U)=0$ or $\nu(\mathcal{G}^{(0)}\setminus U)=0$.
\end{defn}
\begin{no} If $\cG$ is a Borel groupoid with countable $t$-fibers and $(\cG, \nu, \lambda)$ is a measured groupoid such that $\lambda$ is the counting Haar system, then we refer to $(\cG, \nu, \lambda)$ as a \emph{discrete measured groupoid}. Moreover, we will usually drop $\lambda$ from notation and just refer to $(\cG, \nu)$ as a discrete measured groupoid.
\end{no}
\begin{example}
(i) If $G$ is a lcsc group, then $(\underline{G},\delta_*, \lambda^*)$ with $\delta_*$ the Dirac mass on the single object $*$ and $\lambda^{*}$ a left-Haar measure of $G$, is an ergodic measured groupoid. If $G$ is discrete and countable, then $(\underline{G},\delta_*)$ is a discrete measured groupoid. 
\item (ii) Consider a  non-singular (i.e.\ measure class preserving Borel) action $G\curvearrowright (A,\tau)$ of a lcsc group $G$ with left-Haar measure $m_G$ on a Lebesgue space. If we define $\lambda_A^a \coloneqq \delta_a \otimes m_G$ for every $a \in A$, then
$(A \rtimes G, \tau, \lambda_A)$ is a measured groupoid, which is ergodic if and only if the action is ergodic in the usual sense. In particular, if $G$ is a countable discrete group, then $(A \rtimes G, \tau)$ is a discrete measured groupoid.
\item (iii) If $(\cG, \nu, \lambda)$ is a measured groupoid and $Y \subset \cG^{(0)}$ is a subset, then there is in general no way to equip the restriction $\cG|_{Y}$ with the structure of a measured groupoid unless $\nu(Y)>0$, in which case $\cG|_Y$ inherits a structure of measured groupoid by restricting the Haar system and the measure on units.
\end{example}
We are going to discuss more interesting examples of discrete measured groupoids in Subsection \ref{sec:3.2} below.
\begin{defn}\label{def:homomorphism}
  Let $(\mathcal{G}, \nu, \lambda)$ and $(\mathcal{G}',\nu',\lambda')$ be measured groupoids.  
  
\item (i) A \emph{strict homomorphism} between $(\mathcal{G},\nu,\lambda)$ and $(\mathcal{G}',\nu', \lambda')$ is a measurable map $f: \mathcal{G}^{(1)}\to \mathcal{H}^{(1)}$ such that $f_* (\nu\circ \lambda)$
is absolutely continuous with respect to $\nu'\circ\lambda'$, and for every composable pair $(g,h) \in \mathcal{G}^{(2)}$ we have
$(f(g),f(h)) \in \mathcal{H}^{(2)}$ and $f(gh) = f(g)f(h)$.

\item (ii) A strict homomorphism $f: \mathcal{G}^{(1)}\to \mathcal{H}^{(1)}$ is called an \emph{isomorphism} if it is bijective and the inverse map $f^{-1}$ is also a strict homomorphism. In this case we say that $(\mathcal{G},\nu, \lambda)$ and $(\mathcal{G}',\nu', \lambda')$ are \emph{strictly isomorphic}.

\item (iii) We say that $(\mathcal{G},\nu, \lambda)$ and $(\mathcal{G}',\nu', \lambda')$ are \emph{weakly isomorphic} if there exist Borel subsets $A\subset \mathcal{G}^{(0)}$ and $B\subset \mathcal{H}^{(0)}$ of positive measure such that the restriction groupoids $\mathcal{G}|_A$ and $\mathcal{H}|_B$ are strictly isomorphic.
\end{defn}
The following criterion for weak isomorphism of ergodic, principal, discrete measured groupoids due to Furman \cite[Proposition 2.3]{Furman99} will be used in the proof of our main theorem for discrete groups:
\begin{prop}\label{prop:furman:weakly}
  Let $(\mathcal{G},\nu)$ and $(\mathcal{H},\nu')$ be ergodic, principal, discrete measured groupoids. Then $(\mathcal{G},\nu)$ and $(\mathcal{H},\nu')$ are weakly isomorphic
provided that there exist Borel maps $p:\mathcal{G}^{(0)}\to \mathcal{H}^{(0)}$ and $q:\mathcal{H}^{(0)}\to \mathcal{G}^{(0)}$ such that 
  \begin{enumerate}[(i)]
  \item$p_*\nu\prec \nu'$ and $q_*\nu'\prec \nu$;
  \item$(p(t(g)),p(s(g)))\in \mathcal{R}_{\mathcal{H}}$  and $(q(t(h)),q(s(h)))\in \mathcal{R}_{\mathcal{G}}$ for all $g\in \mathcal{G}^{(1)}$, $h\in \mathcal{H}^{(1)}$;
  \item$(q(p(x)),x)\in \mathcal{R}_{\mathcal{G}}$ and $(p(q(y)),y)\in \mathcal{R}_{\mathcal{H}}$ for all $x\in \mathcal{G}^{(0)}$, $y\in \mathcal{H}^{(0)}$.
  \end{enumerate}
\end{prop}

\subsection{Measurable bundles of Banach spaces}\label{sec2.4}

\begin{no} If a group $G$ acts on a set $A$, then it also acts on the set of real-valued functions $\cF(A, \bR)$ on $A$ by $g \cdot f(a) \coloneqq f(g^{-1} \cdot a)$, and the fixpoints are precisely the $G$-invariant functions. On the other hand, if a groupoid $\cG$ acts on a set $A$, then there is a no induced action on functions, since there is no reasonable way to define an anchor map. Nevertheless, to define cohomology, we need to introduce a notion of ``invariant functions''. 
Following \cite{sarti:savini25} we now discuss a suitable framework to discuss invariants in the context of measurable bounded cohomology. We refer also to \cite{fell:doran} for a deeper discussion about these topics.
\end{no}
\begin{no}
Let $(\Omega, \tau)$ be a Lebesgue space and let $\calE=(E_\omega)_{\omega\in \Omega}$ be a family of Banach spaces indexed by $\Omega$. We refer to a map
\[
\sigma: \Omega \to \bigcup_{\omega \in \Omega} E_\omega \quad \text{such that} \quad \sigma(\omega) \in E_\omega \text{ for all }\omega \in \Omega
\]
as a \emph{section of $\cE$}. A collection $\calM$ of sections of $\cE$ is called a \emph{measurable structure} if the following hold:
\begin{enumerate}[(i)]
\item$\sigma_1+\sigma_2\in \calM$ whenever $\sigma_1,\sigma_2 \in \calM$;
\item $\varphi \cdot \sigma \in \calM$ whenever $\sigma\in \calM$ and $\varphi: \Omega \to \bR$ is $\tau$-measurable;
\item if $\sigma\in \calM$ then $\Omega \to \bR$, $\omega \mapsto \lVert \sigma(\omega) \rVert_{E_\omega}$ is $\tau$-measurable;
\item if $(\sigma_n)$ is a net in $\calM$ and $\sigma_n(\omega)\rightarrow \sigma(\omega)$ for $\tau$-almost every $\omega\in \Omega$, then $\sigma \in \calM$.
 \end{enumerate} 
If $\cM$ is a measurable structure, then we refer to $(\cE, \cM)$ as a \emph{measurable bundle (of Banach spaces over $(\Omega, \tau)$)} and to elements of $\cM$ as \emph{measurable sections} of $(\cE, \cM)$. We will often drop $\cM$ from notation and simply talk about measurable section of the measurable bundle $\cE$. 

\item Given $\sigma \in \cM$ we set
\[
\lVert \sigma \rVert:  \Omega \rightarrow [0,\infty), \ \ \ \lVert \sigma \rVert(x)\coloneqq \lVert \sigma(x) \rVert_{E_x}. 
\]
and define the space of \emph{bounded measurable sections} of $\cE$ by
\[\Linf(\Omega,\mathcal{E})\coloneqq \{\sigma\in \mathcal{M}\,|\,  \lVert \sigma \rVert \in\Linf(\Omega) \}/_{\sim_{\tau}}\,,\]
where we identify sections that coincide $\tau$-almost everywhere. This is a Banach space with norm given by 
$\|[\sigma]\|_{\Linf(\Omega,\mathcal{E})} \coloneqq \|\sigma\|$.
 
\item A \emph{morphism} between two measurable bundles $\cE$ and $\cF$ over $(\Omega,\tau)$ is a collection $\varphi=(\varphi_\omega)_{\omega\in \Omega}$ of bounded linear maps $\varphi_\omega:E_\omega \rightarrow F_\omega$ such that
\begin{enumerate}[(i)]
\item for every measurable section $\sigma$ of $\cE$ the map  $\omega\mapsto \varphi_\omega(\sigma(\omega)) $ is a measurable section of $\mathcal{F}$;
\item there exists a constant $C \geq 0$ such that $\| \phi_\omega\|_{\mathrm{op}} < C$ for $\tau$-almost all $\omega$.
\end{enumerate}
Note that every morphism $\phi: \cE \to \cF$ of bundles induces a bounded linear map $\phi_*: \Linf(\Omega,\mathcal{E}) \to \Linf(\Omega,\mathcal{F})$.
 \end{no}
  \begin{con}[Pullback bundle]
  Given a Borel map $\pi:B\rightarrow A$ between Lebesgue spaces $(A,\nu)$ and $(B,\mu)$, a measurable bundle $\calE=(E_a)_{a\in A}$ over $A$ gives rise to a measurable bundle $\mathcal{F}= (F_b)_{b \in B}$ over $B$ by setting
  \[
  F_b=(\pi^\ast E)_b\coloneqq E_{\pi(b)} \quad \text{for all }b \in B.
  \]
  This bundle is denoted by $\pi^*\calE$ and called the  \emph{pullback} of $\calE$ via $\pi$. We equip it with the measurable structure $\pi^*\calM$ generated by the set  
$$\calN\coloneqq \{ \sigma\circ \pi \,,\, \sigma\in \calM\},$$ where
$\calM$ is the measurable structure for $\calE$ \cite[Section 3]{sarti:savini25}, \cite{fell:doran}.
\end{con}
\begin{no}
 From now on,  $(\calG,\nu, \lambda)$ denotes a measured groupoid. We are going to be interested in certain measurable bundles over $(\Omega, \tau) \coloneqq (\cG^{(0)}, \nu)$. Given such a bundle $\cE = (E_y)_{y \in \cG^{(0)}}$ we abbreviate $\underline{\cE} \coloneqq \bigsqcup_{y \in \cG^{(0)}} E_y$; note that this fibers naturally over $\cG^{(0)}$ via the map $t_{\cE}: \underline{\cE} \to \cG^{(0)}$ which send $E_y$ to $y$, and we are interested in certain actions of $\cG$ on $\underline{\cE}$ with anchor map $t_{\cE}$. Note that if $\cG \curvearrowright \underline{\cE}$ is such an action, then for every $v \in E_{s(g)}$  we have $g \cdot v \in E_{t(g)}$.
 \end{no}
 \begin{defn} Let $\cE = (E_y)_{y \in \cG^{(0)}}$ be a measurable bundle over $(\cG^{(0)}, \nu)$. Then a left-action $\cG \curvearrowright \underline{\cE}$ with anchor map $t_{\cE}$ as above if called an
 \emph{isometric left $\calG$-action} on $\mathcal{E}$ if the following properties hold:
\begin{enumerate}[(i)]
\item For every $g\in \mathcal{G}^{(1)}$, the function 
$ E_{s(g)}\rightarrow E_{t(g)}$, $v \mapsto g\cdot v$ is a linear isometry.
\item If $\sigma$ is a measurable section of $\mathcal{E}$, then the map $g\mapsto g\cdot \sigma(s(g))$
is a measurable section of $t^*\calE$ over $(\calG^{(1)}, \nu \circ \lambda)$.
\end{enumerate}
We then refer to $\cE$ together with this action as a \emph{measurable $(\cG, \nu, \lambda)$-bundle}, and to the subspace
\[
\Linf(\cG^{(0)},\mathcal{E})^{\cG} \coloneqq \{[\sigma] \in \Linf(\cG^{(0)},\mathcal{E}) \mid \forall\, g \in \cG^{(1)}:\; \sigma(t(g))=g\cdot \sigma(s(g))\} \subset \Linf(\cG^{(0)},\mathcal{E})
\]
as the \emph{subspace of $\cG$-invariants}.
\end{defn}
 \begin{con}[Measurable bundles from Lebesgue spaces]\label{LBundle}
 Let $t:(A, \tau)\to(Y,\nu)$ be a Borel map between Lebesgue spaces and suppose that $\tau$ disintegrates as $\tau=\int_Y\tau^y d\nu(y)$; we define the \emph{associated measurable bundle} $\cL_{A}$ as follows. 
Given $y \in Y$, we denote by $A^y$ the fiber of $y$ under the anchor map $t_A:A \to Y$. We then obtain a measurable bundle $\cL_A = (\Linf(A^y, \tau^y))_{y \in Y}$ by declaring that a section $\sigma: Y \to \bigcup \Linf(A^y, \tau^y)$ is measurable, provided that the map
\[
\sigma_A: A \to \bR, \quad \sigma_A(a) = \sigma(t_A(a))(a)
\] 
is measurable. Note that if $A$ is a Lebesgue $(\mathcal{G},\nu,\lambda)$-space, the above construction yields a measurable $\mathcal{G}$-bundle: if $g \in \cG^{(1)}$, then $g^{-1}$ defines a map $A^{s(g)} \to A^{t(g)}$, hence given $f \in \Linf(A^{s(g)}, \tau^{s(g)})$ we may define $g \cdot f \in  \Linf(A^{t(g)}, \tau^{t(g)})$ by $g \cdot f(a) \coloneqq  f(g^{-1} \cdot a)$. 
One can check that this defines an isometric left $\cG$-action.
\end{con}
\begin{no} We now return to the problem of defining invariant functions for a Lebesgue $(\cG, \nu, \lambda)$-space $(A, \tau)$. As above we define $A^y \coloneqq  t_A^{-1}(y)$ and given a Borel function $f: A \to \bR$ we denote by $f_y \coloneqq  f|_{A^y}$ the restriction of $A^y$. We then have a natural isomorphism
\begin{equation}\label{fy}
\iota_A: \Linf(A, \tau) \to \Linf(\cG^{(0)}, \cL_A), \quad f \mapsto (f_y)_{y \in \cG^{(0)}},
\end{equation}
which allows us to interpret elements of $\Linf(A, \tau)$ as essentially bounded measurable sections of $\cL_A$. The elements of $\Linf(A, \tau)^{\cG} \coloneqq \iota_A^{-1}(\Linf(\cG^{(0)}, \cL_A)^{\cG})$ can then be interpreted as \emph{invariant function classes}. Explicitly, $[f] \in \Linf(A, \tau)$ is $\cG$-invariant if and only if for every $g \in \cG^{(1)}$ and $\tau^{t(g)}$-almost every $a \in A^{t(g)}$ we have $f(g^{-1}\cdot a) = f(a)$. 
\end{no}
\begin{example}[Tautological bundles]\label{BEBF}
   Let $(\calG,\nu,\lambda)$ be a measured groupoid and let $\cG^{[n]}$ be as in Example \ref{ActionExamples}.(iv); we equip  $\mathcal{G}^{[n]} \subset (\cG^{(1)})^n$ with the induced Borel structure. Given $y \in \cG^{(0)}$ we define $\lambda_{[n]}^y\coloneqq \lambda^y\otimes\cdots\otimes \lambda^y$ so that $(\cG^{[n]}, \nu\circ \lambda_{[n]} )$ is a Lebesgue $\cG$-space. We then refer to the associated measurable bundle $\cL^n\coloneqq \mathcal{L}_{\mathcal{G}^{[n]}}$ as the \emph{$n$-th tautological bundle} of $\cG$. Then a function  in  $\Linf(\mathcal{G}^{[n]}, \nu\circ \lambda_{[n]}) \cong \Linf(\cG^{(0)}, \cL^n)$ is $\cG$-invariant if for all $g \in \cG^{(1)}$ and $\lambda^{t(g)}_{[n]}$-almost all $(g_0, \dots, g_{n-1}) \in (\cG^{t(g)})^n$,
\begin{equation}\label{TautInv}
\phi(g^{-1}g_0,\ldots,g^{-1}g_{n-1}) = \phi(g_0, \dots, g_{n-1}).
\end{equation}
Note that if $(\cG, \nu)$ is a \emph{discrete} measured groupoid, then \eqref{TautInv} holds for all $g \in \cG^{(1)}$ and \emph{all} $(g_0, \dots, g_{n-1}) \in (\cG^{t(g)})^n$.
\end{example}
\begin{no} If $\mathcal{E}$ and $\mathcal{F}$ are measurable $(\cG, \nu, \lambda)$-bundles, then a morphism $\varphi: \cE \to \cF$ is called a 
$\mathcal{G}$-morphism if 
\[\varphi_{t(g)}(g\cdot v)= g\cdot \varphi_{s(g)}(v) \quad \text{for all }g \in \cG^{(1)} \text{ and }v\in E_{s(g)}.\]
This implies that  $\phi_*: \Linf(\cG^{(0)},\mathcal{E}) \to \Linf(\cG^{(0)},\mathcal{F})$ restricts to a map on the level of $\cG$-invariants $\phi_*: \Linf(\cG^{(0)},\mathcal{E})^{\cG} \to \Linf(\cG^{(0)},\mathcal{F})^{\cG}$. Unravelling definitions, one sees that measure class preserving $\cG$-maps between $(\cG, \nu, \lambda)$-spaces induce $\cG$-morphisms of the corresponding bundles and thus well-defined maps on the level of invariants.
\end{no}

\begin{es}[Banach $\cG$-modules and constant bundles]\label{BanachModule}
A \emph{measurable Banach $\mathcal{G}$-module} is a Banach space $E$ together with a map
$$\mathcal{G}^{(1)}\times E\to E\,,\;\;\; (g,v)\mapsto g\cdot v$$ satisfying $g\cdot (h\cdot v)= gh \cdot v$ for every composable pair $(g,h)$ and $v\in E$ and such that the orbital map 
$$\mathcal{G}^{(1)}\to E\,,\;\;\;v\mapsto g\cdot v$$ is measurable for every $v\in E$ (cf.\  \cite{sarti:savini:23}). It is called \emph{isometric} if each of the maps $E \to E$, $v \mapsto g \cdot v$ is an isometry. 
For example, we can turn $\bR$ into an isometric measurable Banach $\cG$-module by declaring every $g \in \cG^{(1)}$ to act by the identity; this is called the \emph{trivial $\mathcal{G}$-module}. More interesting examples will arise naturally in the 
context of transverse measured groupoids (see Construction \ref{PullbackModules} below).

\item Every isometric measurable Banach $\mathcal{G}$-module $E$ gives rise to a measurable $(\cG, \nu, \lambda)$-bundle $\underline{E} = (E)_{y\in \mathcal{G}^{(0)}}$ over $\cG^{(0)}$, called the \emph{constant bundle} with fiber $E$, 
whose measurable sections are just the $\nu$-measurable maps $\mathcal{G}^{(0)}\to E$ and such the isometric left $\cG$-action is given by $\cG^{(1)} \times E \to E$, $(g,v) \mapsto g \cdot v$. By definition we then have
$\Linf(\cG^{(0)}, \underline{E}) = \Linf(\cG^{(0)}, E)$, and we write $\Linf(\cG^{(0)}, E)^{\cG} \coloneqq \Linf(\cG^{(0)}, \underline{E})^{\cG}$. More explicitly, a function class $f$ defines a class in
$\Linf(\cG^{(0)}, E)$ if for all $g \in \cG^{(1)}$ and almost all $v \in E$ we have $f(t(g)) = g \cdot f(s(g))$.
 For the bundle $\underline{\bR}$ corresponding to the trivial module we obtain
$\Linf(\mathcal{G}^{(0)}, \underline{\bR}) =\Linf(\cG^{(0)}, \nu)$.
\end{es}



\subsection{Measurable bounded cohomology of discrete measured groupoids}\label{sec2.5}
Throughout this subsection, $(\cG, \nu, \lambda)$ denotes a measured groupoid with counting Haar system $\lambda$. 
In particular, $\Linf(\mathcal{G}^{[n+1]}, \lambda^y_{[n+1]})$ boils down to the space of bounded functions, that we denote by 
$\ell^\infty((\cG^y)^{n+1})$.
\begin{con}[Homogeneous bar resolution] 
As in Example \ref{BEBF}, denote by $\cL^n \coloneqq   (\ell^\infty((\cG^y)^{n+1}))_{y\in \mathcal{G}^{(0)}}$ the tautological bundles over $\cG^{(0)}$ and denote by $\underline{\bR}$ the trivial bundle. There is a $\cG$-morphism $d_{\cG}^{-1} = (d_{y}^{-1})_{y \in \cG^{(0)}}: \underline{\bR} \to \cL^0$ such that $d_y^{-1}: \bR \to \ell^\infty(\cG^y)$ is the inclusion of constants. Moreover, for every $n \in \bN_0$ there is a $\cG$-morphism $d_{\cG}^n = (d_{y}^{n})_{y \in \cG^{(0)}}: \cL^n \to \cL^{n+1}$ such that
\begin{gather*}
  d_y^n: \ell^\infty((\cG^y)^{n+1}) \to  \ell^\infty((\cG^y)^{n+2}), \\ d_y^nf(g_0, \dots, g_{n+1}) \coloneqq    \sum_{j=0}^{n+1} (-1)^j f(g_0, \dots, \widehat{g_j}, \dots, g_{n+1}).
\end{gather*}
On the level of bounded measurable sections the sequence of $\cG$-bundles and $\cG$-morphisms
\begin{equation}\label{StdRes}
\underline{0} \to \underline{\bR} \xrightarrow{d_{\cG}^{-1}} \cL^0 \xrightarrow{d_{\cG}^{0}} \cL^1 \to \dots,
\end{equation}
induces an exact complex of dual Banach spaces and weak-$*$-continuous maps, called the \emph{augmented homogeneous bar resolution} for $\cG$, and given by
\[
0 \to  \Linf(\cG^{(0)},\underline{\bR}) \xrightarrow{d_{\cG}^{-1}} \Linf(\cG^{(0)}, \cL^0) \xrightarrow{d_{\cG}^{0}} \Linf(\cG^{(0)}, \cL^1) \to \dots,
\]
or, more explicitly,
\begin{equation*}\label{AugmentedStandardResolution}
0 \to \Linf(\cG^{(0)}) \xrightarrow{d_{\cG}^{-1}}  \Linf(\cG^{[1]})  \xrightarrow{d_{\cG}^0} \Linf(\cG^{[2]})  \xrightarrow{d_{\cG}^1} \Linf(\cG^{[3]})  \xrightarrow{d_{\cG}^2}  \dots
\end{equation*}
where $d_{\cG}^{{-1}}$ is given by $d_{\cG}^{-1}(f)(g_0) = f(t(g_0))$ and for $n \geq 0$ we have
\[
d_{\cG}^n f(g_0, \dots, g_{n+1}) = \sum_{j=0}^{n+1} f(g_0, \dots, \widehat{g_j}, \dots, g_{n+1}).
\]
The subcomplex obtained by deleting the augmentation, i.e.
\begin{equation}\label{BarResolution}
0 \to \Linf(\cG^{[1]})  \xrightarrow{d_{\cG}^0} \Linf(\cG^{[2]})  \xrightarrow{d_{\cG}^1} \Linf(\cG^{[3]})  \xrightarrow{d_{\cG}^2}  \dots,
\end{equation}
is called the \emph{homogeneous bar resolution} for $\cG$. Since the differentials are induced by $\cG$-morphisms, we can pass to the subcomplex 
\[
0 \to \Linf(\cG^{[1]})^{\cG}  \xrightarrow{d_{\cG}^0} \Linf(\cG^{[2]})^{\cG}  \xrightarrow{d_{\cG}^1} \Linf(\cG^{[3]})^{\cG}  \xrightarrow{d_{\cG}^2}  \dots
\]
of $\cG$-invariants. By definition, the cohomology of this complex is the measurable bounded cohomology of $(\cG, \nu)$ with coefficients in $\underline{\bR}$:
\end{con}
\begin{defn} The \emph{measurable bounded cohomology} of $(\cG,\nu, \lambda)$ with coefficients in the trivial line bundle $\underline{\bR}$ over $\cG^{(0)}$ is the seminormed space given by
\[
\Hmb^k((\cG,\nu, \lambda); \underline{\bR}) \coloneqq  \mathrm{H}^k(\Linf(\cG^{[\bullet + 1]})^{\cG}, d_{\cG}^{\bullet}) \quad (k \geq 0)
\]
with the seminorm induced by the essential supremum norm on cochains.
\end{defn}
It is easy to see that measurable bounded cohomology is invariant under strict isomorphism of measured groupoids; more remarkable is the fact that for \emph{ergodic} measured groupoids it is also invariant under weak isomorphism \cite{sarti:savini:23}. 
\begin{prop}\label{proposition:weak:iso:BC}
Let $(\cG,\nu,\lambda)$ and $(\mathcal{H},\nu',\lambda')$ be weakly isomorphic ergodic measured groupoids. Then there is an isomorphism
\[\Hmb^\bullet((\cG,\nu,\lambda);\underline{\mathbb{R}})\cong\Hmb^\bullet((\mathcal{H},\nu',\lambda');\underline{\mathbb{R}}).\]
\end{prop}
\begin{no}\label{FunctorialCharacterization} It would be desirable to have a functorial characterization of measurable bounded cohomology of measured groupoids along the lines of \cite{BuhlerThesis, monod:libro}; however, this seems to be beyond the reach of current technology. On the other hand, for \emph{discrete} measured groupoid such a characterization was established by Savini and the second-named author in \cite{sarti:savini25}. Since we will use it in the proof of our main theorem, let us give a brief overview.

Thus let $(\cG, \nu)$ be a discrete measured groupoid. In \cite{sarti:savini25}, the authors introduce the notion of strong augmented resolutions of a $(\cG, \nu)$-bundle by relatively injective measurable $(\cG, \nu)$-bundles. They then show that \eqref{StdRes} is a strong augmented resolution of $\underline{\bR}$ by relatively injective measurable bundles, and that any such resolution can be used to compute measurable bounded cohomology. More precisely they prove that if $\mathcal{A}^\bullet$ and  $\mathcal{B}^\bullet$ are two strong resolutions of $\underline{\bR}$ by relatively injective measurable $(\mathcal{G},\nu)$-bundles, then the following hold:

\item (i) There exists $\cG$-morphisms which makes the following diagram commute:
\[\begin{xy}\xymatrix{
0 \ar[r]  & \underline{\mathbb{R}} \ar[r] \ar@{=}[d]&\mathcal{A}^0  \ar[r] \ar[d] & \mathcal{A}^1 \ar[r] \ar[d]&\dots\\
0 \ar[r]  & \underline{\mathbb{R}}  \ar[r] & \mathcal{B}^0  \ar[r] & \mathcal{B}^1 \ar[r]&\dots\\
}
\end{xy}\]
\item (ii) Any choice of $\cG$-morphisms as in (i) induces a commutative ladder
\[\begin{xy}\xymatrix{
0 \ar[r] & \Linf(\cG^{(0)}, \mathcal{A}^0)^{\cG}  \ar[r] \ar[d] & \Linf(\cG^{(0)},\mathcal{A}^1)^\cG \ar[r] \ar[d]&\dots\\
0 \ar[r] & \Linf(\cG^{(0)},\mathcal{B}^0)^\cG  \ar[r] & \Linf(\cG^{(0)},\mathcal{B}^1)^{\cG} \ar[r]&\dots\\
}
\end{xy}\]
of cochain complexes of Banach spaces and the induced maps
\[
\Hm^k( \Linf(\cG^{(0)}, \mathcal{A}^\bullet)^{\cG} ) \to \Hm^k( \Linf(\cG^{(0)}, \mathcal{B}^\bullet)^{\cG} ) \quad (k \geq 0).
\]
are natural isometric isomorphisms and independent of the choice of $\cG$-morphisms. 
\end{no}
\begin{remark} We emphasize a subtlety of the above functorial characterization: For the maps in (ii) to induce isometric isomorphisms, we need to assume that they are induced from equivariant bundle maps as in (i). 
There is currently no functorial characterization of measurable bounded cohomology of discrete groupoids in terms of maps between the corresponding Banach spaces of sections only.
\end{remark}

\section{Transverse measured groupoids}\label{sec:trans:system}
Throughout this section, $G$ denotes a unimodular lcsc group and $m_G$ denotes a fixed choice of Haar measure for $G$.

\subsection{Transverse systems and transverse measures}\label{sec3.1}

\begin{no} Consider a probability measure preserving (pmp) Borel action of $G$ on a standard Borel probability space $(X,\mu)$. Given a Borel subset $Y \subset X$ and $x \in X$ we define the associated \emph{hitting time set} $$Y_x \coloneqq \{g \in G \mid g x \in Y\}\,;$$
   we also define the \emph{return time set} of $Y$ as
\[
\Lambda(Y) \coloneqq \bigcup_{y \in Y} Y_y = \{g \in G \mid g y \cap Y \neq \emptyset\}.
\]
We recall from the introduction that a Borel subset $Y \subset X$ is called a \emph{cross section} if $GY= X$ and for every $x \in X$ the \emph{hitting time set} $Y_x$ is locally finite in the sense that for every compact subset $L \subset G$ we have $|Y_x \cap L| < \infty$. We then refer to $(X, \mu, Y)$ as a \emph{transverse $G$-system}. In this case, each of the sets $Y_x$ is countable, since $G$ is $\sigma$-compact; however, in general the subsets $Y_x \subset G$ need not be discrete, and $\Lambda(Y)$ need not even to be countable.
\end{no}
\begin{example}[Integrable transverse systems] Let $Y \subset X$ be a Borel subset such that $G Y = X$ and such that the point process $x \mapsto Y_x$ is locally integrable, i.e.\
\begin{equation}\label{IntegrabilityDef}
\int_X |Y_x \cap L| \, d \mu(x) < \infty \quad \text{for every compact subset }L \subset G.
\end{equation}
This implies that for every compact $L \subset G$ the intersection $Y_x \cap L$ is finite for almost all $x$; since $G$ is $\sigma$-compact this implies that the subset
\[
X_0 = \{x \in X \mid Y_x \text{ is locally finite}\} \subset X
\]
is conull, and one can check that it is Borel. Since $Y_{gx} = Y_xg^{-1}$ for all $g \in G$ and $x \in X$ it is also $G$-invariant, hence if we define $Y_0 \coloneqq X_0 \cap Y$, then $Y_x = (Y_0)_x$ for all $x \in X_0$. This shows that 
$(X_0, \mu|_{X_0}, Y_0)$ is a transverse $G$-system whose associated point process is locally integrable; we refer to such a system as an \emph{integrable transverse system}.
\end{example}
\textbf{From now on, $(X, \mu, Y)$ denotes a transverse $G$-system.}
\begin{lemma}\label{CountFib1} The action map $a:G\times Y\to X$, $a(g,y) \coloneqq gy$ has countable fibers.
\end{lemma}
\begin{proof} If $x \in X$ and $(g,y) \in  a^{-1}(x)$, i.e.\ $g y = x$, then
\[
g \in Y_x^{-1} \text{ and }y = g^{-1}  x \implies (g,y) \in  Y_x^{-1} \times Y_x  x.
\]
The lemma then follows since $Y_x$ is countable.
\end{proof}

\begin{con}[Cocycles over transverse systems]\label{ConCocycle} 
By Lemma \ref{CountFib1} the action map $a: G \times Y \to X$ admits a Borel section
\begin{equation}\label{bfix}
b:X \to G\times Y \quad \text{such that} \quad b(y) = (e,y) \; \text{ for all }y \in Y.
\end{equation}
Indeed, the existence of a section follows by the Lusin-Novikov selection theorem (see e.g.\ \cite[Lemma 2.1]{BHK} for the version needed here) and the normalization can be arranged since $Y$ is Borel.
We now fix such a Borel section once and for all and define a Borel function
\[
\beta: G \times Y \to G, \quad \beta(x) \coloneqq  (\mathrm{pr}_G (b(x)))^{-1}.
\]
By definition we then have $b(x) = (\beta(x)^{-1}, \beta(x)x)$, and $\beta$ satisfies
\begin{equation}\label{BetaNormalization}
\beta(x) \in Y_x \qand \beta(y) = e \quad \text{for all }x \in  G \times Y \text{ and }y \in Y.
\end{equation}
We deduce that if $x \in X$ and $g \in G$, then $\beta(gx)g\beta(x)^{-1}$ is a return time of $Y$ which 
sends $\beta(x)  x \in Y$ to $\beta(g x)  g x \in Y$. We thus obtain a Borel map
 \[ 
    \sigma:G\times X\to \Lambda(Y)\,,\;\;\; \sigma(g,x)\coloneqq \beta(g x)g\beta(x)^{-1},
\]
which is a \emph{normalized cocycle} in the sense that
\begin{equation}\label{CocycleId}
\sigma(g_1g_2, x) = \sigma(g_1, g_2  x)\sigma(g_2, x) \quad \text{and} \quad \sigma(e,x) = e \quad (g_1, g_2 \in G, x \in X).
\end{equation}
\begin{center}
  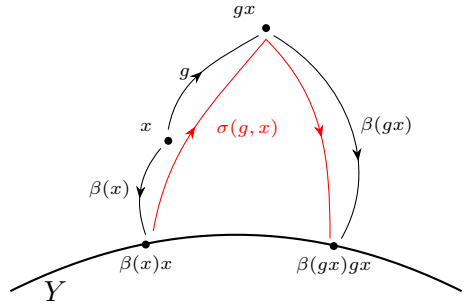
\begin{figure}[!ht]
  \begin{tikzpicture}
    [roundnode/.style={circle, black, minimum size=2pt}]
\draw[thick] (-3,-2) .. controls (-1,-1) and (1,-1) .. node[pos=0.1, anchor=north] {$Y$} (3,-2);
\fill (-1.2,-1.38) circle (1.5pt) node[below] {\tiny  $\beta(x)x$};
\fill (-0.9,0) circle (1.5pt) node[above left] {\tiny $x\;$};
\fill (0.4,1.5) circle (1.5pt) node[above left] {\tiny $gx$};
\fill (1.3,-1.4) circle (1.5pt) node[below] {\tiny $\beta(gx)gx$};
\draw[{Butt Cap[] Stealth[sep=0.5cm]}-] (-1.2,-1.25) to [out=110,in=-130] node[pos=0.5, anchor=east] {\tiny $\beta(x)$} (-1,-0.1) ;
\draw[{Butt Cap[] Stealth[reversed,sep=0.8cm]}-] (-0.9, 0.15) to [out=80,in=-150] node[pos=0.5, anchor=east] {\tiny $g$} (0.3,1.4) ;
\draw[{Butt Cap[] Stealth[reversed,sep=2cm]}-] (0.5,1.4) to [out=-30,in=60] node[pos=0.5, anchor=west] {\tiny $\beta(gx)$} (1.40,-1.3) ;
\draw[red,{Butt Cap[] Stealth[reversed,red,sep=1.5cm]}-] (-1.1,-1.2) to [out=80,in=-130] node[pos=0.5, anchor=west] {\tiny $\;\;\sigma(g,x)$} (0.4,1.35) ;
\draw[red,{Butt Cap[] Stealth[reversed,red,sep=1.5cm]}-] (0.4,1.35) to [out=-45,in=90](1.25,-1.3) ;
\end{tikzpicture}\caption{The cocycle $\sigma$.}
  \end{figure}
\end{center}
\end{con}
\begin{prop}[{Refined Campbell theorem, \cite[Proposition\ 4.2]{ABC}}]\label{Campbell} There exists a unique $\sigma$-finite measure $\nu$ on $Y$ such that for every Borel function $w: G \to [0, \infty]$ with $m_G(w) = 1$ and every Borel function $f: Y \to [0, \infty]$ we have
\begin{equation}
\int f \, d \nu = \int_X \sum_{g \in Y_x} f(g x) w(g) \, \, d \mu(x).
\end{equation}
Moreover, for every Borel function $F: G \times Y \to [0, \infty]$ we have the Campbell formula
\begin{equation}\label{equation:formula:integral}
\int_{G \times Y} F \, d(m_G \otimes \nu) = \int_X \sum_{g \in Y_x} F(g^{-1}, g x)\, d\mu(x).
\end{equation}
\end{prop}
\begin{defn} The measure $\nu$ from Proposition \ref{Campbell} is called the \emph{transverse measure} of $(X, \mu, Y)$.
\end{defn}
In the probabilistic literature the transverse measure of $(X, \mu, Y)$ is also called the \emph{Palm measure} of the point process $x \mapsto Y_x$. The term ``transverse measure'' is justified by the following proposition; here $a: G \times Y \to X$ denotes the action map.
\begin{prop}[{\cite[Proposition\ 4.9]{ABC}}]\label{TransverseMeasure} If $C \subset G \times Y$ is a Borel set such that $a|_C$ is injective, then $a_*((m_G \otimes \nu)|_C) = \mu|_{a(C)}$.
\end{prop}
\begin{no}\label{CFPOU}  We say that a a countable family $(C_n)_{n \in \bN}$ of pairwise disjoint Borel sets in $G \times Y$ is an \emph{injective cover} if $X$ if the action map restricts to an injection on each $C_n$ and if
$(a(C_n))_{n \in \bN}$ is a Borel partition of $X$. It was established in \cite[Lemma 4.6]{ABC} (extending \cite[\S 3.1]{BHK}) that injective covers always exist in our present setting. We fix such a cover $(C_n)_{
n \in \bN}$ once and for all. 

\item As a first application of injective covers, let us show that the transverse measure $\nu$ determines $\mu$: For every $n \in \bN$ we then define a $\sigma$-finite Borel measure $\mu_n$ on $C_n$ by  $\mu_n \coloneqq (a|_{C_n})^{-1}_\ast (\mu|_{a(C_n)})$. Then, by definition of an injective cover, $\mu$ is uniquely determined by the sequence $(\mu_n)$, and by Proposition \ref{TransverseMeasure} we have
\[
\mu_n = (m_G \otimes \nu)|_{C_n},
\]
hence $\nu$ determines $\mu$.

\item As a second application, we observe that for every $x \in X$ there exists a unique $n \in \bN$ and a unique $g \in Y_x$ such that $(g^{-1}, g x) \in C_n$, i.e.
\[
 \sum_{g \in Y_x}  \sum_{n=1}^\infty 1_{C_n}(g^{-1}, g x) = 1
\]
Thus if we define
\[
\rho: G \times X \to [0,1], \quad \rho(g,x) \coloneqq \sum_{n=1}^\infty 1_{C_n}(g,x)
\]
then for all $x \in X$ we have
\begin{equation}\label{DefPOU}
\sum_{g \in Y_x} \rho(g^{-1}, g x) = 1.
\end{equation}
A function with this property is called a \emph{Borel partition of unity} for $(X,Y)$.
\end{no}
\begin{prop} If $\rho: G \times X \to [0,1]$ is a Borel partition of unity then for every Borel function $\phi: X \to [0, \infty]$ we have
 \begin{equation}\label{POU}
\int_X \phi(x) \, d \mu(x) = \int_G \int_Y \phi(g y) \rho(g,y) \, d\nu(y)\, dm_G(g) \quad (\phi \in \Linf(X, \mu)).
\end{equation}
\end{prop}
\begin{proof} Apply the Campbell formula \eqref{equation:formula:integral} to the function $F: G \times Y \to [0, \infty]$ by $F(g,y) \coloneqq \phi(g y)\rho(g,y)$ to see that the right-hand side of \eqref{POU} equals
\begin{align*}
\int_X \sum_{g \in Y_x} F(g^{-1}, g x)\, d\mu(x) &= \int_X \sum_{g \in Y_x} \phi(g^{-1}g x)\rho(g^{-1}, g x)\, d \mu(x)\\
& = \int_X \phi(x) \left(\sum_{g \in Y_x} \rho(g^{-1}, g x) \right)\, d \mu(x) = \mu(\phi).\qedhere
\end{align*}
\end{proof}
\begin{prop}[{\cite[Proposition\ 5.1]{ABC}}] The transverse measure $\nu$ of the $G$-system $(X, \mu, Y)$ is finite if and only if $(X,\mu, Y)$ is integrable.
\end{prop}

\subsection{From transverse systems to discrete measured groupoids}\label{sec:3.2}
Throughout this subsection, $(X, \mu, Y)$ denotes a transverse $G$-system with transverse measure $\nu$.
\begin{con}[Transverse measured groupoid]
By Example \ref{ResBorel}, the left-action groupoid $G \ltimes X$ and its restriction $\cG \coloneqq (G \ltimes X)|_Y$ are Borel groupoids. Explicitly, the object space and morphism space of $\cG$ are respectively given by
\begin{equation}\label{TrGrpExpl}
\cG^{(0)} = Y \qand \cG^{(1)} = \{(g,y) \in G \times Y \mid g \in Y_y\},
\end{equation}
and the structure maps are given by $s(g,x) = x$, $t(g,x) = g x$ and
\[
(g,x)(h,y) =  (gh, y) \quad \text{if } x  = h  y.
\]
The latter implies in particular that $(g,x)^{-1} = (g^{-1}, g x)$. It is immediate from \eqref{TrGrpExpl} that the $t$-fibers are given by
\[
\cG^y =  t^{-1}(y) = \{(g^{-1}, g y) \mid g \in Y_y\} \quad (y \in \cG^{(0)}),
\]
and hence the map $Y_y \to \cG^y$, $g \mapsto (g^{-1}, g y)$ is a bijection for every $y \in \cG^{(0)}$. This shows that the groupoid $\cG$ has countable $t$--fibers, hence admits a canonical Haar system $\lambda$ given by counting measures along the fibers. 

The key observation for our purposes is that the transverse measure $\nu$ is $(\cG, \lambda)$-invariant or, equivalently, invariant under $R_{\cG}$ \cite[Thm.\ 4.3]{ABC}, and hence the pair $(\cG, \nu)$ is a discrete measured groupoid. We refer to $(\cG, \nu)$ as the \emph{transverse measured groupoid} of $(X, \mu, Y)$.
\end{con}
\begin{remark} The transverse measured groupoid can be defined for any transverse system, but in this generality the transverse measure $\nu$ will only be $\sigma$-finite. All of our main results will rely crucially on finiteness of $\nu$ or, equivalently, integrability of $(X, \mu, Y)$.
\end{remark}
\begin{example}[The periodic case]\label{LatticeCase} 
  Let $\Gamma < G$ be a lattice. We denote by $\mu$ the unique $G$-invariant probability measure on $X \coloneqq \Gamma \backslash G$ and set $Y = \{e\Gamma\}$. Then $(X, \mu, Y)$ is an integrable transverse system over $G$, the associated transverse groupoid is isomorphic to the groupoid $\underline{\Gamma} = (\Gamma \rightrightarrows \{e\Gamma\})$, and the transverse measure is the Dirac measure of total mass $\text{covol}(\Gamma)^{-1}$.
\end{example}
\subsection{Induced actions and induced Banach modules}\label{sec:3.3}
Throughout this subsection,  $(X, Y, \mu)$ denotes an integrable transverse system over $G$ with transverse measured groupoid $(\cG, \nu)$.

\begin{no} There is a natural morphism of Borel groupoids $i_{\cG}: \cG \to \underline{G}$, where $\underline{G} = G \rightrightarrows \{*\}$ is the Borel groupoid defined by $G$. Indeed, this morphism maps every object of $\cG$ to $*$ and maps $(g,y) \in \cG^{(1)}$ to $g$. Note that since $g \in Y_y$, the image of $\cG^{(1)}$ is contained in $\Lambda(Y) \subset G$. We will show that we can use this morphism to 
\begin{itemize}
\item induce a Borel $\cG$-spaces from every Borel $G$-spaces;
\item pullback Banach $G$-modules to Banach $\cG$-modules.
\end{itemize}
Moreover, the two constructions are compatible. We now turn to the details.
\end{no}
\begin{con}[Induced action] \label{ConStarAction}
Let $A$ be a Borel $G$-space; we use the morphism $i_{\cG}$ to induce an action of $\cG$ on $A \times Y$ as follows: We define the anchor map $A \times Y \to Y$ to be the projection to the second factor and define the action map by
\begin{equation}\label{GeneralStarAction}
\cG^{(1)} \times_Y (A \times Y) \to A \times Y, \quad (g,y) \star (a,y) \coloneqq  (g\cdot a, g y).
\end{equation}
Since $i_{\cG}$ is a morphism, this does indeed define an action, and thus $(A \times Y, \star)$ becomes a Borel $\cG$-space, called the $\cG$-space \emph{induced} from $A$. Similarly, if $(A, \tau)$ is a Lebesgue $G$-space, then $(A \times Y, \tau \otimes \nu, \star)$ is a Lebesgue $(\cG, \nu, \lambda)$-space. 
\end{con}
\begin{con}[Pullback module]\label{PullbackModules} If $E$ is a Banach-$G$-module, then $E$ becomes a measurable Banach $\cG$-module in the sense of Example \ref{BanachModule} via \[(g,y) \cdot v \coloneqq g \cdot v;\] again, this follows from the fact that $i_{\cG}$ is a morphism. In particular, we can then consider the associated constant $\cG$-bundle $\underline{E} = (E)_{y \in \cG^{(0)}}$.
\end{con}
\begin{no}\label{InducedBundle} The previous two constructions are related as follows. Let $(A, \tau)$ be a Lebesgue $G$-space and denote by $E \coloneqq \Linf(A,\tau)$ the associated Banach-$G$-module. By Construction \ref{ConStarAction} we obtain a Lebesgue $\cG$-space $(A \times Y, \tau \otimes \nu, \star)$. This gives rise to a measurable $\cG$-bundle $\cL_{A \times Y}$ as in Construction \ref{LBundle}. With notation as in this construction we have
$(A \times Y)^y = A \times \{y\}$ for every $y \in \cG^{(0)}$ and the corresponding fiber measure is just $\tau \otimes \delta_y$. 

\item Given a Borel function $f: A \times Y\to \bR$ and $y \in Y$, we denote $f_y: A \to \bR$, $f_y(a) \coloneqq f(a,y)$. We then obtain a bundle isomorphism
\[
\phi:  \cL_{A \times Y} \to \underline{E}, \quad [f|_{A \times \{y\}}] \mapsto [f_y].
\]
In particular, the $\cG$-bundle associated with an induced space is (isomorphic to) a constant bundle. As a consequence, we see that the classes in $\Linf(A \times Y)^{\cG}$ are those represented by functions $f: A \times Y \to \bR$ such that
for all $(g,y) \in \cG^{(1)}$ and almost all $a \in A$ we have
$f(a, g y) = f(g^{-1} \cdot a, y)$.
\end{no}
\begin{example}\label{StarActionG} We consider $A = G$ with the $G$-action given by $\gamma \cdot a \coloneqq \gamma g^{-1}$. The induced $\cG$-action on $G \times Y$ as given by
\[
(g,y) \star (g', y) = (\gamma g^{-1}, gy).
\]
then happens to commute with the $G$-action on the first factor by  \emph{left}-multiplication. We thus obtain a $(\cG \times G)$-action 
\begin{equation}\label{MEAction}
\cG \times G \curvearrowright G \times Y, \quad ((g,y), k) \star (\gamma , y) = (k\gamma g^{-1}, gy).\tag{$\star$}
\end{equation}
\end{example}

\begin{center}
  \begin{figure}[!ht]
  \begin{tikzpicture}[scale=0.98]
\draw[thick] (-7,-2) .. controls (-5,-1) and (-3,-1) .. node[pos=0.9, anchor=north] {$Y$} (-1,-2);
\fill (-4.9,0) circle (1.5pt) node[above right] {\tiny $\gamma  y\;$};
\fill[red] (-4.0,-1.24) circle (1.5pt) node[below] {\color{red} \tiny $y\;$};

\draw[ red,{Butt Cap[] Stealth[sep=0.5cm]}-] (-4.8,0) to [out=-30,in=80] node[pos=0.5, anchor=west] {\color{red}\tiny $\gamma $} (-4.0,-1.14) ;

\draw[thick] (1,-2) .. controls (3,-1) and (5,-1) .. node[pos=0.9, anchor=north] {$Y$} (7,-2);
\fill (3.1,0) circle (1.5pt) node[above right] {\tiny $\gamma  y\;$};
\fill (4,-1.24) circle (1.5pt) node[below] {\tiny $y\;$};
\fill[red] (5.3,-1.4) circle (1.5pt)node[above] {\color{red} \tiny $gy\;$};
\fill (2.2,1) circle (1.5pt) node[above] {\tiny $k \gamma  y\;$};
\draw[{Butt Cap[] Stealth[sep=0.5cm]}-] (3.2,0) to [out=-30,in=80] node[pos=0.5, anchor=west] {\tiny $\gamma $} (4,-1.14) ;
\draw[ {Butt Cap[] Stealth[reversed, sep=0.5cm]}-] (4.1,-1.4) to [out=-70,in=-110] node[pos=0.5, below] {\tiny $g$} (5.25,-1.54) ;
\draw[{Butt Cap[] Stealth[reversed,sep=0.5cm]}-] (3.1,0.1) to [out=80,in=-30] node[pos=0.5, anchor=south west] {\tiny $k$} (2.3,1) ;

\draw[red,{Butt Cap[] Stealth[reversed,sep=0.5cm]}-] (3.1,0) to [out=120,in=-30] (2.2,0.9) ;
\draw[red,{Butt Cap[] Stealth[sep=0.5cm]}-] (3.1,0) to [out=-30,in=120] node[pos=0.5, anchor=east] {\tiny $k\gamma  g^{-1}$}(4,-1.24);
\draw[red,{Butt Cap[] Stealth[sep=0.5cm]}-] (4,-1.24) to [out=-30,in=-160]  (5.15,-1.52);

\draw[thick,->] (-1,1) to [out=30,in=150] node[pos=0.5, above] { $((g,y),k)\star$}   (1,1);
\end{tikzpicture}\caption{The action $\star$.}
  \end{figure}
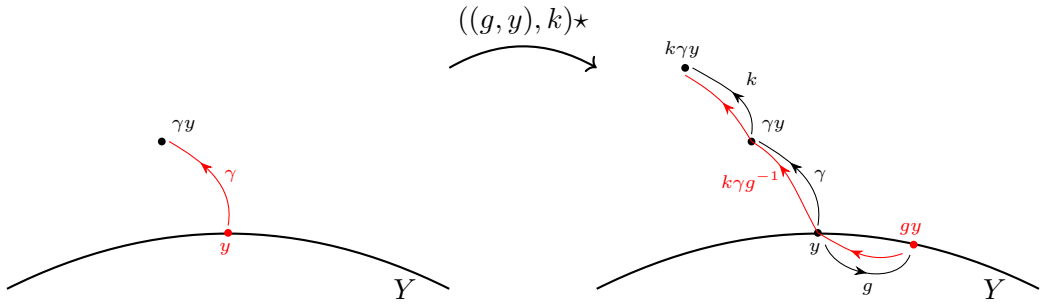
\end{center}

The following generalization of Example \ref{StarActionG} will be crucial for us in the sequel:
\begin{example}
If $A$ is a Borel $G$-space, then we can define a whole family of $(\cG \times G)$-spaces $A^n \times G$ by
 \begin{equation}\label{StarAction}
 \cG \times G \curvearrowright A^n \times G \times Y, \quad ((g,y), k) \star_A (\underline{a}, \gamma , y) = (g\cdot \underline{a}, k \gamma g^{-1}, gy),\tag{$\star_A$}
\end{equation}
with anchor map given by the projection to the final factor.
\end{example}

\section{Measurable bounded cohomology and measure equivalence}\label{sec:BCandME}

We now digress to prove a version of our main theorem for (ergodic) transverse groupoids over countable \emph{discrete} groups $G$. This version has a clean conceptual proof in terms of measure equivalence, which is not available in the non-discrete case due to current technical limitations in the theory of non-discrete Borel groupoids. Nevertheless, the proof in the discrete case should be a seen as a blueprint for the general case. The key observation is that measure equivalent discrete measured groupoids have closely related bounded cohomology, as first established by Monod and Shalom \cite{MonShal} in the group case and recently reproved by the second-named author and Savini \cite{sarti:savini:23} in the groupoidal language.

\subsection{Couplings and measure equivalences}\label{sec:4.1}

In this subsection we extend the notion of coupling and measure equivalence from discrete groups to discrete measured groupoids. Throughout, $(\cG, \nu_{\cG})$, $(\cH, \nu_{\cH})$ and $(\cK, \nu_{\cK})$ denote discrete measured groupoids.
\begin{no}\label{QuotientMeasure} Let $\Omega$ be a $\cG \times \cH$-Borel space and assume that the $\cH$-action admits a strict fundamental domain $Y \subset \Omega$. Since the $\cG$-action commutes with the $\cH$-action, it descends to an action $\cG \curvearrowright \cH \backslash \Omega$.

Moreover, for every $(g,y)\in\mathcal{G}^{(1)}\times_{\mathcal{G}^{(0)}} Y$, there exists a unique element $\alpha(g,y) \in \mathcal{H}$ such that $(g,\alpha(g,y)) \cdot y\in Y$, and this defines a groupoid homomorphism $\alpha: \mathcal{G}\ltimes Y \to \mathcal{H}$ called the \emph{coupling cocycle} of the pair $(\Omega, Y)$. In particular, $\cG$ acts on $Y$ via \begin{equation}\label{DiamondAction}g \diamond  y= (g, \alpha(g,y)) \cdot y.\end{equation} Moreover, the composition $Y \hookrightarrow \Omega \twoheadrightarrow \cH \backslash \Omega$ defines a Borel isomorphisms which intertwines the $\cG$ actions on $Y$ and $\cH \backslash \Omega$. 

It follows from \eqref{DiamondAction} that if $\Omega$ admits a $\cG \times \cH$-invariant measure $m$, then $m|_Y$ is invariant under $\cG$, and hence corresponds to a $\cG$-invariant measure $\bar m$ on $\cH \backslash \Omega$; we refer to the latter as the \emph{quotient measure} of $m$ on $\cH \backslash \Omega$.
\end{no}

\begin{defn}\label{DefCoupling}
A \emph{$(\mathcal{G},\mathcal{H})$-coupling} is a Lebesgue space $(\Omega,m)$ with a free measure preserving $\mathcal{G}\times\mathcal{H}$-action admitting finite measure strict fundamental domains $X,Y\subset \Omega$ for the $\mathcal{G}$-action and the $\mathcal{H}$-action respectively. It is called \emph{ergodic} if the action groupoid $(\mathcal{G} \times \mathcal{H})\ltimes \Omega$ (equivalently, the underlying equivalence relation $\mathcal{R}_{(\mathcal{G} \times \mathcal{H})\ltimes \Omega}$) is ergodic. We say that $\cG$ and $\cH$ are \emph{measure equivalent} if there exists a $(\cG, \cH)$-coupling.
\end{defn}
\begin{no} If $(\Omega, m)$ is a $(\cG, \cH)$-coupling, then we write 
$\mathcal{G}\overset{\textup{ME}}{\sim}_{(\Omega, m)} \mathcal{H}$
and say that $G$ and $H$ are ME \emph{via $(\Omega, m)$}. We also write $\mathcal{G}\ME \mathcal{H}$ to indicate that $\cG$ and $\cH$ are measure equivalent. Note that this implies that there exists an \emph{ergodic} $(\cG, \cH)$-coupling. If we consider discrete groups as discrete measured groupoids with one object, then Definition \ref{DefCoupling} restricts to the usual definitions of measure equivalence of discrete groups \cite{furman:survey}.
\end{no}
\begin{prop}\label{OME} Measure equivalence is an equivalence relation on the class of discrete measured groupoids.
\end{prop}
\begin{proof} Compared to the group case, the main subtlety arises from the fact that $\nu$ need not be invariant, and hence  $\nu \circ \lambda$ is in general not equal but only equivalent to its image $(\nu\circ \lambda)^{-1}$ under the inversion map. This can be resolved by replacing $\nu \circ \lambda$ by its symmetrization $(\nu\circ \lambda) +(\nu\circ \lambda)^{-1}$ and then the proof is as in the group case:

A $(\cG, \cG)$-coupling is given by $(\Omega, m) \coloneqq (\mathcal{G}^{(1)},(\nu\circ \lambda) +(\nu\circ \lambda)^{-1})$ with the $(\cG \times \cG)$-action given by
\[
(g_1,g_2)\cdot  h \coloneqq g_1hg_2^{-1} \quad \text{for all }h\in \mathcal{G}^{(1)}, g_1 \in \mathcal{G}_{t(h)}, g_2\in \mathcal{G}_{s(h)}
\]
This shows reflexivity, and symmetry is immediate since any $(\mathcal{G},\mathcal{H})$-coupling $(\Omega,m)$ can be seen as $(\mathcal{H},\mathcal{G})$-coupling by simply swapping the actions. 

As for transitivity, assume that we are given a $(\mathcal{G},\mathcal{H})$-coupling $(\Omega_1,m_1)$ with fundamental domains $X$ and $Z_1$ and a $(\mathcal{H},\mathcal{K})$-coupling $(\Omega_2,m_2)$ with fundamental domains $Z_2$ and $Y$. We now define $\Omega' \coloneqq \Omega_1 \times_{\cH^{(0)}} \Omega_2$; then $\cH$ acts diagonally on $\Omega'$, and there are commuting actions of $\cG$ and $\cK$ on the factors. We thus obtain a $(\cG \times \cK)$-action on $\Omega \coloneqq \cH \backslash \Omega'$ which preserves the quotient measure $m$ of $(m_1 \otimes m_2)|_{\Omega'}$. We now claim that $(\Omega, m)$ is the desired $(\mathcal{G},\mathcal{K})$-coupling. 
Indeed, a strict $\mathcal{G}$-fundamental domain is given by the set 
$$\widetilde{X}\coloneqq \{ [x,\omega] \mid x\in X\,,\, \omega\in \Omega_2\,,\, t^{\mathcal{H}}_{\Omega_1}(x)=t^{\mathcal{H}}_{\Omega_2}(\omega) \},$$
and since $\mathcal{H}$ acts trivially on the first factor of $\widetilde{X}$ we have
$${m}(\widetilde{X})=m_1(X)m_2(\mathcal{H}\backslash \Omega_2)=m_1(X)m_2(Z_2)< +\infty.$$ 
Similarly, a $\mathcal{H}$-fundamental domain is given by
$$\widetilde{Y}\coloneqq \{ [\omega,y] \mid \omega \in \Omega_1\,,\, y\in Y \,,\, t^{\mathcal{H}}_{\Omega_1}(\omega)=t^{\mathcal{H}}_{\Omega_2}(y)\}\,.$$ 
and we have  \[m(\widetilde{Y})=m_1(\mathcal{H}\backslash\Omega_1)m_2(Z_2)=m_1(Z_1)m_2(Y) < +\infty .\qedhere\]
\end{proof}
\begin{example}\label{example:ME}
Let $\mathcal{G}$ be a discrete measured groupoid with counting Haar system $\lambda$ and $(A,\tau)$ be a Borel space of finite measure on which $\mathcal{G}$ acts in a measure preserving way. 
We claim that $\cG \ME \cG\ltimes A$. More precisely, we have
\[
\cG \ME_{((\mathcal{G}\ltimes A)^{(1)}, \tau \circ \lambda)} \cG \ltimes A,
\]
where the action of $\cG \ltimes A$ is the tautological one, and the $\cG$-action has anchor map $t(g,a)\coloneqq t_A(a)$ and is given by
$$h\circ  (g,a)\coloneqq (gh^{-1},h\cdot a)\, \quad \text{for every }h \in \mathcal{G}_{t(g,a)}.$$ 
Indeed, a fundamental domain of finite measure for both actions is given by 
$$\mathcal{G}^{(0)}\times_{\mathcal{G}^{(0)}} A=\{ (t_A(a),a)\,,\, a\in A \}\,.$$
\end{example}

\subsection{Measurable bounded cohomology under couplings}\label{sec:4.2}
As before, $(\cG, \nu_{\cG})$ and $(\cH, \nu_{\cH})$ denote discrete measured groupoids.
\begin{no}\label{OEcocycle}
Let $(\Omega, m)$ be a $(\mathcal{G},\mathcal{H})$-coupling with fundamental domains $X, Y \subset \Omega$. We denote by $\alpha: \mathcal{G}\ltimes Y \to \mathcal{H}$ and $\beta: \cH \ltimes X \to \cG$ the corresponding coupling cocycles as  so that the induced actions  $\cG \curvearrowright \cH\backslash \Omega \cong Y$ and $\cH \curvearrowright \cG \backslash \Omega \cong X$ are given by
  \begin{equation}\label{eq:action:ME}
     g \diamond  y= (g, \alpha(g,y)) \cdot y \qand  h \bullet  x= (\beta(h,x), h)  x.
  \end{equation}
  As observed in \S \ref{QuotientMeasure} these actions preserve the corresponding restrictions of $m$, and hence we obtain two new discrete measured groupoids
\begin{equation}\label{NewGroupoids}
(\cG \ltimes Y,  m|_Y) \qand (\cH \ltimes X,  m|_X).
\end{equation}
\end{no}
\begin{lemma} If the coupling $(\Omega, m)$ is ergodic, then so are the groupoids in \eqref{NewGroupoids}.
\end{lemma}
\begin{proof} 
Let $U\subset Y = (\cG \ltimes Y)^{(0)}$ be a $\mathcal{G}$-invariant Borel subset and consider the $\mathcal{H}$-invariant Borel set $\widetilde{U}\coloneqq \mathcal{H}U$. Since
$\mathcal{G}U=U$, commutativity of the actions yields
$$\mathcal{G} \widetilde{U}= \mathcal{G}\mathcal{H}U=\mathcal{H} \mathcal{G}U=\mathcal{H}U=\widetilde{U}\,,$$
hence $\widetilde{U}$ is $(\cG \times \cH)$-invariant.
Thus $\widetilde{U}$ is $\mathcal{G}$-invariant. By ergodicity we thus have either $m(\widetilde{U})=0$ or $m(\Omega\setminus\widetilde{U})=0$. In the first case, we obtain 
$m|_{Y}(U)=m(U)\leq m(\widetilde{U})=0$. Similarly in the second case we get
\[m|_{Y}(Y\setminus U)=m(Y\setminus U)\leq m(\mathcal{H}(Y\setminus U))=m(\Omega\setminus\widetilde{U})=0,\] whence $\mathcal{G}\ltimes Y$ is ergodic. 
The proof for $\mathcal{H}\ltimes X$ is obtained by reversing the roles.\end{proof}
We now have the following key lemma:
\begin{lemma}\label{lemma:iso}
  Let $(\Omega,m)$ be an ergodic $(\mathcal{G},\mathcal{H})$-coupling with respective fundamental domains $X$ and $Y$. Then $\mathcal{G}\ltimes Y$ and $\mathcal{H}\ltimes X$ are weakly isomorphic ergodic, principal, discrete measured groupoids.
\end{lemma}
  \begin{proof}
We can mimic the same construction done by Monod and Shalom in the group case \cite{MonShal} (compare also \cite{sarti:savini:23}) 
  For a given $u\in \Omega$, we define subsets
  $$\mathcal{G}u\coloneqq \left\{g\cdot u \mid g\in \mathcal{G}_{t_\Omega^\mathcal{G}(u)}\right\}\subset \Omega\,
 \qand \mathcal{H}u\coloneqq \left\{h\cdot u \mid h\in \mathcal{H}_{t_\Omega^\mathcal{H}(u)}\right\}\subset \Omega\,.$$
Since $X$ and $Y$ are strict fundamental domains, we then obtain maps $p:Y\to X$ and $q: X \to Y$ such that
$\mathcal{G}y\cap X = \{p(y)\}$ and $\mathcal{H}x\cap Y = \{q(x)\}$. We now check that these maps satisfy Furman's criterion (Proposition \ref{prop:furman:weakly}); note that this criterion applies since $\mathcal{G}\ltimes Y$ and $\mathcal{H}\ltimes X$ are ergodic by the previous lemma and principal by freeness of the action.
\begin{enumerate}[(i)]
  \item[(i)]$p_*(m|_Y)\prec m|_X$ and $q_*(m|_X)\prec m|_Y$: this follows since $\mathcal{G}$ and $\mathcal{H}$ act in a measure preserving way.
  \item[(ii)] $(p(g\diamond  y),p(y))\in \mathcal{R}_{\mathcal{H}\ltimes X}$ for all $y\in Y$ and $g\in \mathcal{G}_{t_\Omega^{\mathcal{G}}}$ and $(q(h\bullet  x),p(x))\in \mathcal{R}_{\mathcal{G}\ltimes Y}$ for all $x\in X$ and $h\in \mathcal{H}_{t_\Omega^{\mathcal{H}}}$: this follows by the definition of the actions $\diamond$ and $\bullet$ and because the $\mathcal{G}$ and $\mathcal{H}$-actions on $\Omega$ commute.
  \item[(iii)]  $(q(p(y)),y)\in \mathcal{R}_{\mathcal{G}\ltimes Y}$ for all $y\in Y$ and $(p(q(x)),x)\in \mathcal{R}_{\mathcal{H}\ltimes X}$ for all $x\in X$: this holds by definition.\qedhere
\end{enumerate}
  \end{proof}
  Combining this with Proposition \ref{proposition:weak:iso:BC} we obtain the following extension of \cite[Corollary 6.1]{sarti:savini:23}.
\begin{cor}\label{METwist}
 Let $(\Omega,m)$ be an ergodic $(\mathcal{G},\mathcal{H})$-coupling with fundamental domains $X,Y\subset \Omega$ for $\mathcal{G}$ and $\mathcal{H}$ respectively. Then we have an isomorphism
\[\pushQED{\qed}\Hmb^\bullet((\mathcal{G}\ltimes Y,  m|_Y); \underline{\mathbb{R}})\cong \Hmb^\bullet((\mathcal{H}\ltimes X,  m|_X);\underline{\mathbb{R}}). \qedhere\popQED\]
\end{cor}
\begin{example}\label{RecoverSaSa} Given an ergodic pmp action $G \curvearrowright (Y, \mu)$ of a countable discrete group $G$ we have
\[
\Hmb^k((G\ltimes Y,\mu);\underline{\mathbb{R}})\cong \Hb^k(G; \Linf(Y)).
\]
As a special case of Corollary \ref{METwist} we thus recover \cite[Theorem 4.6]{MonShal}: If two countable discrete groups $G, H$ are measure equivalent via an ergodic coupling $(\Omega, \mu)$ with respective fundamental domains $X$ and $Y$, then
\[
\Hb^k(G; \Linf(Y)) \cong \Hb^k(H; \Linf(X)).
\]
\end{example}

\subsection{Application to transverse groupoids}\label{sec:4.3}
In this subsection, let $G$ be a countable discrete group with counting measure $m_G$ and let $(X, \mu, Y)$ be an ergodic integrable transverse $G$-system with transverse measured groupoid $(\cG, \nu)$. The following was observed by Bj\"orklund and the first named author (unpublished), motivated by questions about the uniqueness of envelopes for approximate groups.
\begin{prop}\label{METrans} Assume that $G$ is a countable discrete group. Then $(G \times Y, m_G \otimes \nu)$ with the $(G \times \cG)$-action given by \eqref{MEAction} defines an ergodic coupling between
the discrete measured groupoids $\underline{G}$ and $\cG$ with respective fundamental domains isomorphic to $(X, \mu)$ and $(Y, \nu)$.
\end{prop}
\begin{proof} The fact that $m_G \otimes \nu$ is invariant under  $G \times \cG$ is immediate from $\cG$-invariance of $\nu$ and unimodularity of $G$, and ergodicity follows from ergodicity of $\mu$.

 Moreover, since $G$ acts trivially on the second factor we have $G \backslash (G \times Y) \cong Y$, and the $\cG$-action on the quotient corresponds to the standard action on $Y$. 
 
It is a routine verification that $\cG$ acts transitively on the fibers of the $G$-equivariant suspension map $S: G \times Y \to X$, $(g,y) \mapsto g^{-1}y$, and hence $\cG \backslash (G \times Y) \cong X$ as $G$-spaces; moreover, under this identification, the restriction of $m_G \otimes \nu$ to a Borel fundamental domain for $\cG$ corresponds to the $G$-invariant measure $\mu$ on $X$. 
\end{proof}
\begin{cor}\label{MERig} Let $G$ and $H$ be countable discrete groups. If there is an integrable transverse $G$-system with transverse groupoid $(\cG, \nu_{\cG})$ and an integrable transverse $H$-system with transverse groupoid $(\cH, \nu_{\cH})$ such that  $(\cG, \nu_{\cG})$ and  $(\cH, \nu_{\cH})$ are strictly isomorphic, then $G$ and $H$ are measure equivalent.\qed
\end{cor}
\begin{remark}[The non-discrete case] If $G$ is a non-discrete lcsc group, then we still have commuting measure-preserving actions of $G$ and $\cG$ on $(G \times Y, m_G \otimes \nu)$ via  \eqref{MEAction}, and the suspension map $S: G \times Y \to X$ still induces a $G$-equivariant Borel isomorphism $\cG \backslash (G \times Y) \cong X$ with countable fibers such that the restriction of $m_G \otimes \nu$ to a Borel fundamental domain for $\cG$ corresponds to $\mu$. However, the $G$-action on $G \times Y$ no longer has countable fibers. Nevertheless, there is a still a (more technical) notion of measure equivalence in this context so that $G$ and $\cG$ are measure equivalent, and hence Corollary \ref{MERig} extends to arbitrary lcsc groups $G$ and $H$. Combining this with the powerful rigidity results of Furman \cite{Furman99} that higher rank Lie groups are essentially determined by the isomorphism class of any of their transverse groupoids. We refer to forthcoming work of Bj\"orklund and the first named author for details.
\end{remark}
Proposition \ref{METrans} now implies a version of our main theorem under the additional assumptions that $G$ is discrete and the action is ergodic:
\begin{cor}\label{MainThmDiscrete} If $G$ is a countable discrete group and $(X, \mu, Y)$ is an ergodic integrable transverse $G$-system with transverse measured groupoid $(\cG, \nu)$, then there is an isomorphism
\[
\Hmb^\bullet((\mathcal{G}, \nu); \underline{\bR}) \cong \Hb^\bullet(G; \Linf(X, \mu)).
\]
\end{cor}
\begin{proof} Since $(\cG \ltimes Y,  \nu) \cong (\cG, \nu)$ this follows by combining Proposition \ref{METrans}, Corollary \ref{METwist} and Example \ref{RecoverSaSa}.
\end{proof}
\begin{remark}
Our main theorem says that Corollary \ref{MainThmDiscrete} also holds for general lcsc groups $G$ and without the ergodicity assumption. One would like to deduce the results for lcsc groups from a more general version of Proposition \ref{METrans} for non-discrete Borel groupoids. However, the technical tools for such a general result are currently missing; we will thus go down a different path and prove the non-discrete version of Corollary \ref{MainThmDiscrete} in a more direct way, following an approach of Ph.\ Blanc. As an added bonus, we can get rid of the ergodicity assumption. Nevertheless it would be interesting to know whether Proposition \ref{METrans} holds for non-discrete measured groupoids.
\end{remark}

\section{Transfer isomorphisms}\label{sec:trans:isom}
In this section we introduce the technical tools which will be needed to extend Corollary \ref{MainThmDiscrete} to the non-discrete case. Throughout this section, $G$ denotes a unimodular lcsc group $G$ with fixed choice of Haar measure $m_G$ and
 $(X, Y, \mu)$ denote a separated transverse system with transverse measured groupoid $(\cG, \nu)$. We fix a Borel map $\beta: X \to G$ and cocycle $\sigma: G\times X\to \Lambda(Y)$ as in Construction \ref{ConCocycle}.
\subsection{Transfer spaces}\label{sec:trans:space}
\begin{con}[Transfer space]\label{ConIso} We define the  \emph{transfer space} for the pair $(\cG, G)$ as
 the fiber product
\[\cG^{(1)}\times_{Y} X \coloneqq  \cG^{(1)} {}_A\times_{p_\beta} X \subseteq \cG^{(1)} \times X,\]
where $p_\beta: X \to Y$ is defined by $p_\beta(x) \coloneqq  \beta(x) x$.
If $(h,y) \in \cG$ and $x\in X$, then we have
\[
(h,y,x) \in \cG^{(1)}\times_{Y} X \iff \beta(x)x = y,
\]
and hence $y$ is redundant. We thus introduce the short-hand notation
\[
[h,x] \coloneqq  (h, \beta(x)x, x).
\]
Note that for every $x \in X$ we have $\beta(x)x \in Y$, hence $(h, \beta(x)x) \in \cG^{(1)}$ if and only if $h\beta(x)x \in Y$, i.e.\
\begin{equation}\label{FiberProdRed}
\cG^{(1)}\times_{Y} X = \{[h,x] \mid h \in G, x \in X,  h\beta(x)x \in Y\}.
\end{equation}
In particular, $\cG^{(1)} \times_Y X$ embeds as a subset of $\Lambda(Y) \times X$ via $[h,x] \mapsto (h,x)$.
\end{con}
\begin{con}[Commuting actions]\label{CommutingActions} The groupoid $\cG$ acts on the transfer space $\cG^{(1)} \times_Y X$ as follows: The anchor map is given by $[h,x] \mapsto h\beta(x)x$, and  if $y  = h\beta(x)x$, then the action is defined by
\[
(g, y)\ast [h,x] \coloneqq [gh, x].
\]
We now define a commuting action of $G$ on  $\cG^{(1)} \times_Y X$ using the Borel map $\beta$ and cocycle $\sigma$: If $[h,x] \in \cG^{(1)}\times_{Y} X$, then by \eqref{FiberProdRed} we have
\[
h\sigma(k,x)^{-1}\beta(kx)kx = h( \beta(kx)k\beta(x)^{-1})^{-1}\beta(kx)kx = h\beta(x)x \in Y,
\]
hence  $ [h\sigma(k,x)^{-1}, kx] \in \cG^{(1)}\times_{Y} X$. We may thus define a $G$-action by
\[
k \ast [h,x] \coloneqq [h\sigma(k,x)^{-1}, kx].
\]
By construction, this action commutes with the $\cG$-action, hence we obtain a $(\cG \times G)$-action on the transfer space given by
\begin{equation}\label{AstAction1}
((g, y), k) \ast [h,x] \coloneqq ([gh\sigma(k,x)^{-1}, kx]). \tag{$\ast$}
\end{equation}
\end{con}
\begin{center}
  \begin{figure}[!ht]
  \begin{tikzpicture}[scale=0.98]
\draw[thick] (-7,-2) .. controls (-5,-1) and (-3,-1) .. node[pos=0.9, anchor=north] {$Y$} (-1,-2);
\fill (-5.2,-1.38) circle (1.5pt) node[below left=1pt and -4pt] {\tiny  $\beta(x)x$};
\fill[blue] (-4.9,0) circle (1.5pt) node[above] {\color{blue}\tiny $x\;$};
\fill (-4.0,-1.24) circle (1.5pt);
\draw[ {Butt Cap[] Stealth[sep=0.5cm]}-] (-5.2,-1.25) to [out=110,in=-130] node[pos=0.5, anchor=east] {\tiny $\beta(x)$} (-5,-0.1) ;
\draw[ {Butt Cap[] Stealth[reversed, sep=0.5cm]}-,color=blue] (-5.2,-1.5) to [out=-70,in=-110] node[pos=0.5, below] {\tiny $h$} (-4.0,-1.34) ;

\draw[thick] (1,-2) .. controls (3,-1) and (5,-1) .. node[pos=0.9, anchor=north] {$Y$} (7,-2);
\fill (2.8,-1.38) circle (1.5pt) node[below left=1pt and -4pt] {\tiny  $\beta(x)x$};
\fill (3.1,0) circle (1.5pt) node[right] {\tiny $x\;$};
\fill (4.0,-1.24) circle (1.5pt);
\fill (1.5,-1.77) circle (1.5pt) node[below=1.5 pt] {\tiny $\;\;\beta(kx)kx$};
\fill[blue] (2.1,1.5) circle (1.5pt) node[above left] {\color{blue} \tiny $kx$};
\fill (5.3,-1.4) circle (1.5pt);

\draw[ {Butt Cap[] Stealth[sep=0.5cm]}-] (2.8,-1.25) to [out=110,in=-130] node[pos=0.5, anchor=east] {\tiny $\beta(x)$} (3,-0.1) ;
\draw[ {Butt Cap[] Stealth[reversed, sep=0.5cm]}-] (2.8,-1.5) to [out=-70,in=-110] node[pos=0.5, below] {\tiny $h$} (4.0,-1.34) ;
\draw[ {Butt Cap[] Stealth[reversed, sep=0.5cm]}-] (3.1,0.2) to [out=90,in=-20] node[pos=0.5, anchor=west] {\tiny $k$} (2.2,1.4) ;
\draw[ {Butt Cap[] Stealth[reversed, sep=1.5cm]}-] (1.9,1.4) to [out=-130,in=110] node[pos=0.5, anchor=east] {\tiny $\beta(kx)$} (1.4,-1.67) ;
\draw[ {Butt Cap[] Stealth[reversed, sep=0.5cm]}-] (4.03,-1.4) to [out=-70,in=-110] node[pos=0.5, below] {\tiny $g$} (5.3,-1.54) ;

\draw[blue, {Butt Cap[] Stealth[reversed, sep=1cm]}-] (1.55,-1.57) to [out=110,in=-120] (2.1,1.3);
\draw[blue,{Butt Cap[] Stealth[reversed, sep=0.5cm]}-] (2.1,1.3) to [out=-80,in=100] (3.1,0);
\draw[blue,{Butt Cap[] Stealth[reversed, sep=0.5cm]}-] (3.1,0) to [out=-100,in=60] node[right] {\color{blue} \tiny $gh\sigma(k,x)^{-1}$} (2.8,-1.38);
\draw[blue,{Butt Cap[] Stealth[reversed, sep=0.5cm]}-] (2.8,-1.38) to [out=-60,in=-120] (4.0,-1.24);
\draw[blue,{Butt Cap[] Stealth[reversed, sep=0.5cm]}-] (4.0,-1.24) to [out=-60,in=-120] (5.2,-1.45);

\draw[thick,->] (-1,1) to [out=30,in=150] node[pos=0.5, above] { $((g,y),k)\ast$}   (1,1);
\end{tikzpicture}\caption{The action $\ast$.}
  \end{figure}
\end{center}

\begin{prop}\label{psiIso} There is a well-defined Borel isomorphism
\begin{equation}\label{psi}
  \psi: G\times Y\to \mathcal{G}^{(1)} \times_Y X\,,\;\;\; \psi(\gamma,y)=[\gamma^{-1}\beta(\gamma y)^{-1}, \gamma y]\end{equation}
 with inverse given by
  $$\phi: \mathcal{G}^{(1)} \times_Y X \to G\times Y\,,\;\;\; \phi([h,x])=(\beta(x)^{-1} h^{-1} , h \beta(x) x).$$
 Moreover, this isomorphism intertwines the induced $(\cG \times G)$-action action on $G\times Y$ with the action on the transfer space given by \eqref{AstAction1}.
\end{prop}
\begin{proof} We first show that $\psi$ is well-defined. Since $s(\beta(\gamma y), \gamma y) = t(\gamma, y)$ we can form the product
\[
(\beta(\gamma y), \gamma y)(\gamma, y) = (\beta(\gamma y)\gamma, y),
\]
which has source $y$ and target $\beta(\gamma y)\gamma y \in Y$, hence is contained in $\cG^{(1)}$. Now pass to the inverse to see that $((\beta(\gamma y), \gamma, y)(\gamma, y))^{-1} \in \cG^{(1)}$; this shows that $\psi$ is well-defined, and it is Borel as a composition of Borel maps. Similarly, $\phi$ is well-defined by \eqref{FiberProdRed} and Borel by definition. We compute 
  \begin{align*}
    \psi \circ \phi([h,x]) &= \psi(\beta(x)^{-1} h^{-1},h \beta(x) x)\\
     & = [h \beta(x) \beta(h \beta(x)\beta(x)^{-1} h^{-1}x)^{-1} , \beta(x)^{-1} h^{-1}h\beta(x) x ] = [h,x]
  \end{align*}
  and dually
  \begin{align*}
  \phi\circ \psi(\gamma,y)&= \phi ([\gamma^{-1}\beta( \gamma y)^{-1},\gamma y])\\
 &= (\beta(\gamma y)^{-1}\beta( \gamma y) \gamma , \gamma^{-1}\beta(\gamma y)^{-1} \beta(\gamma y)\gamma y)=(\gamma,y),
  \end{align*}
  which shows that $\psi$ and $\phi$ are mutually inverse Borel isomorphisms. To see equivariance we observe that, on the one hand,
\begin{align*}
    \psi(((g,y),k)\star (\gamma,y))&=\psi(k\gamma g^{-1},gy) =([g\gamma^{-1}k^{-1} \beta(k\gamma y)^{-1},k\gamma y]),
\end{align*}
and on the other hand,
\begin{align*}
   ((g,y),k)  \ast \psi(\gamma,y)&=   ((g,y),k)\ast  ([\gamma^{-1}\beta(\gamma y)^{-1},\gamma y])\\
   &=[g(\gamma^{-1}\beta(\gamma y)^{-1})\sigma(k,\gamma y)^{-1},k\gamma y].
\end{align*}
Spelling out the definition of $\sigma$ now yields
\[
g(\gamma^{-1}\beta(\gamma y)^{-1})\sigma(k,\gamma y)^{-1} = g\gamma^{-1}\beta(\gamma y)^{-1}(\beta(k \gamma y)k \beta(\gamma y)^{-1})^{-1}
= g \gamma^{-1}k^{-1}\beta(k\gamma y)^{-1},
\]
which shows that the two expressions are equal. This concludes the proof.
\end{proof}
\begin{example}
  In the lattice case of Example \ref{LatticeCase}, we have
\[
G \times Y \cong G \qand  \mathcal{G}^{(1)} \times_Y X = \mathcal{G}^{(1)} \times X \cong \Gamma \times \Gamma \backslash G.
\]
and $\psi$ boils down to the Borel isomorphism
$$G\cong\Gamma\times \Gamma\backslash G\,,\;\;\; g\mapsto (\sigma(g^{-1},\Gamma g),\Gamma g)$$
described in  the proof of \cite[Lemma 5.4.3]{monod:libro}. 
\end{example}
\begin{no} We will need the following slightly technical extension of Proposition \ref{psiIso}. Here, given a a Borel $G$-space $A$ and $n \in \bN$, we extend the $(\cG \times G)$-action on the transfer space to an action on $A^n \times \cG^{(1)}\times_Y X$
by
\begin{gather}\label{AstAction}
\cG \times G \curvearrowright A^n \times \cG^{(1)}\times_Y X, \tag{$\ast_A$}\\ ((g,y),k) \ast_A (\underline{a},[h,x])\coloneqq (\sigma(k,x)\cdot \underline{a}, [gh\sigma(k,x)^{-1}, kx]).\notag
\end{gather}
\end{no}
\begin{cor}\label{IncludeS} For every Borel $G$-space $A$ and every $n \in \mathbb N_0$ there is a Borel isomorphism
$\eta_{n}: A^{n+1}\times \mathcal{G}^{(1)}\times_Y X \to A^{n+1}\times G\times Y$ given by 
\begin{equation}\label{Defetan}
\eta_{n}(\underline{a}, [h,x])\coloneqq  (h \cdot \underline{a}, \varphi([h,x]))= (h\cdot \underline{a},\beta(x)^{-1} h^{-1},h\beta(x)x),
\end{equation}
which is equivariant for the  $(\cG \times G)$-actions given by \eqref{StarAction} and \eqref{AstAction}.
Moreover, a Borel inverse 
$\xi_n: A^{n+1}\times G\times Y  \to    A^{n+1}\times \mathcal{G}^{(1)}\times_Y X$  for $\eta_n$ is given by
\begin{equation}\label{Defxin}
  \xi_n(\underline{a},\gamma,y)\coloneqq (\beta(\gamma y)\gamma\cdot \underline{a},\psi(\gamma,y)) =(\beta(\gamma y)\gamma \cdot \underline{a},[\gamma^{-1}\beta(\gamma y)^{-1},\gamma y])\,.
\end{equation}
\end{cor}
\begin{proof} Bijectivity of is immediate from bijectivity of $\varphi$ and $\psi$. Precisely, we have
\begin{align*}
  \eta_n(\xi_n(\underline{a},\gamma,y))&=\eta_n(\beta(\gamma y)\gamma \cdot \underline{a},[\gamma^{-1}\beta(\gamma y)^{-1},\gamma y])\\
  &= (\underline{a}, \gamma,y)
\end{align*}  
and 
\begin{align*}
  \xi_n(\eta_n(\underline{a}, [h,x]))&=\xi_n(h\cdot \underline{a},\beta(x)^{-1} h^{-1},h\beta(x)x)\\
  &= (\underline{a}, [h,x])\,.
\end{align*} 
The computation for equivariance is similar as in the proof of Proposition \ref{psiIso}: On the one hand we have
\begin{align*}
    \xi_n(((g,y),k)\star_A (\underline{a}, \gamma,y))&=\xi_n(g\cdot \underline{a}, k\gamma g^{-1},gy)\\
    &=(\beta(k\gamma y)k\gamma  \cdot \underline{a}, \psi(k\gamma g^{-1},gy))\\
    &=(\beta(k\gamma y)k\gamma \cdot \underline{a}, [g\gamma^{-1}k^{-1} \beta(k\gamma y)^{-1},k\gamma y]),
\end{align*}
and on the other hand,
\begin{align*}
   ((g,y),k)  \ast_A \xi_n(\underline{a}, \gamma,y)&=   ((g,y),k)\ast_A  (\beta(\gamma y)\gamma \cdot \underline{a},[\gamma^{-1}\beta(\gamma y)^{-1},\gamma y])\\
   &=(\sigma(k,\gamma y)\beta(\gamma y)\gamma \cdot \underline{a},[g(\gamma^{-1}\beta(\gamma y)^{-1})\sigma(k,\gamma y)^{-1},k\gamma y])\,.
\end{align*}
We have seen in the proof of Proposition \ref{psiIso} that the second components are equal. Moreover, exploiting the definition of $\sigma$ we can conclude that 
$$\sigma(k,\gamma y)\beta(\gamma y)=\beta(k\gamma y)k\beta(\gamma y)^{-1}\beta(\gamma y)=\beta(k\gamma y)k$$
hence also the first components coincide.
This concludes the proof.
\end{proof}
\subsection{Transfer measures}\label{sec:trans:meas}
\begin{con}[Transfer measure] On $G \times Y$ there is a canonical measure class defined by the infinite measure $m_G \otimes \nu$, where $\nu$ denotes the transverse measure of $\mu$. We now pick a partition of unity $\rho$ for $(X,Y)$ (cf. \S \ref{CFPOU}); then by \eqref{POU} the measure $\rho \cdot m_G \otimes \nu$ is a probability measure representing this measure class. Via the Borel isomorphism $\psi$ from \eqref{psi} this representative corresponds to a probability measure  $\mu_\rho$ on the fiber product $\mathcal{G}^{(1)} \times_Y X$. We refer to $\mu_\rho$ as the \emph{transfer measure} defined by $\mu$ and $\rho$. By construction we have:
\end{con}
 \begin{cor}\label{CorEta1} For every Lebesgue $G$-space $(A, \tau)$ and every $n \in \mathbb N_0$ the map $\eta_n$ from \eqref{Defetan} defines a pmp isomorphism of Lebesgue $(\cG \times G)$-spaces
 \[
  (A^{n+1}\times \mathcal{G}^{(1)}\times_Y X, \tau^{\otimes n+1} \otimes {\mu}_\rho) \cong (A^{n+1}\times G\times Y , \rho \cdot \tau^{\otimes n+1} \otimes m_G \otimes \nu).
 \]
 \end{cor}
 We have the following explicit formula for $\mu_\rho$:
\begin{lemma}\label{murho} If $F$ is a non-negative bounded Borel function on $\mathcal{G}^{(1)} \times_Y X$, then
 \[
 \int_{\mathcal{G}^{(1)} \times_Y X} F \; d{\mu}_\rho = \int_X \sum_{g \in Y_x} F([g\beta(x)^{-1}, x])   \rho(g^{-1}, gx)\; d \mu(x).
 \]
 \end{lemma}
 \begin{proof} Since $\mu_\rho =  \psi_*(\rho \cdot m_G\otimes \nu)$ it follows from \eqref{equation:formula:integral} that
 \begin{align*}
  \int_{\mathcal{G}^{(1)} \times_Y X} F\; d\mu_{\rho} &= \int_{G} \int_Y (F \circ\psi)(g,y) \rho(g,y) \, d \nu(y) \, d m_G(g)\\
  &= \int_X T((F \circ \psi) \cdot \rho)(x) d\mu(x).
 \end{align*}
 We can now write out the integrand explicitly:
 \begin{align*}
 T((F \circ \psi) \cdot \rho)(x) &=  \sum_{\gamma \in Y_x}(F \circ \psi)(\gamma^{-1}, gx) \rho(\gamma^{-1}, gx)\\
 &= \sum_{g \in Y_x}  F([g \beta(x)^{-1} ,x])  \rho(g^{-1}, gx).\qedhere
 \end{align*}
 \end{proof}
 \begin{cor}\label{PushRhoAway} If $\pi:  \mathcal{G}^{(1)}\times_Y X \to X$, $[h,x]  \mapsto x$ denotes the canonical projection, then $\pi_*\mu_{\rho} = \mu$.
\end{cor}
\begin{proof} Let $f$ be a bounded Borel function on $X$. By Lemma \ref{murho} we then have
\begin{align*}
(\pi_*\mu_\rho)(f) &=  \int_X \sum_{g \in Y_x} f(x)  \rho(g^{-1}, gx)\; d \mu(x) = \int_X f(x)  \sum_{g \in Y_x} \rho(g^{-1}, gx)\; d \mu(x)\\
&= \int_X f(x) T\rho(x) \; d\mu(x) =  \int_X f \; d \mu(x) = \mu(f).\qedhere
\end{align*}
\end{proof}

\subsection{Transfer isomorphisms}\label{sec:isomo}
\begin{no}\label{TwistedAction} From now on we fix a Lebesgue $G$-space $(A,\tau)$ and some $n \in \bN_0$. We then consider $A^{n+1} \times Y$ as a $\cG$-space with respect to the induced action.
There are two natural $G$-actions on $(A^{n+1} \times X,\tau^{\otimes n+1}\otimes \mu)$, namely the untwisted (i.e.\ diagonal) action given by
$k\cdot (\underline{a},x)=(k\cdot \underline{a},kx)$, and the \emph{$\sigma$-twisted action} given by
\begin{equation}\label{eq:twisted:action}
k\ast (\underline{a},x)=(\sigma(k,x)\cdot \underline{a},kx)\,,
\end{equation}
For distinction, we denote the respective $G$-spaces by $A^{n+1} \times X$ and $A_{\sigma}^{n+1} \times X$.
\end{no}
\begin{thm}[Transfer isomorphism]\label{BanachIso} There is an isometric isomorphism
\[
j^{n}_A: \Linf(A_\sigma^{n+1} \times X, \tau^{\otimes n+1}\otimes \mu)^G \to \Linf(A^{n+1} \times Y, \tau^{\otimes n+1}\otimes \nu)^{\cG},
\]
which on the level of representatives if given by restriction to $A^{n+1} \times Y$.
\end{thm}
\begin{remark} Since all measure classes involved in Theorem \ref{BanachIso} are the canonical ones (given $\mu$), we will often drop them from notation and simply write
\[
j^{n}_A: \Linf(A^{n+1}_\sigma \times X)^G \to  \Linf(A^{n+1} \times Y)^{\cG}.
\]
We emphasize that the map $j^n_A$ is only defined on the level of invariants, i.e. there is no well-defined map  restriction map $\Linf(A^{n+1} \times X) \to \Linf(A^{n+1} \times Y)$.
\end{remark}
\begin{no} The transfer isomorphism $j^n_A$ involves the twisted $G$-space $A^{n+1}_\sigma \times X$. One can also construct a transfer isomorphism into the untwisted $G$-space $A^{n+1} \times X$, by using the twisting isomorphism
  \begin{equation}
  \varrho^A_n:(A^{n+1}\times X,\tau^{\otimes n+1}\otimes \mu)\to (A_{\sigma}^{n+1}\times X,\tau^{\otimes n+1}\otimes \mu), \quad (\underline{a},x) \mapsto (\beta(x)\cdot \underline{a},x)
\end{equation}
and its inverse as given by $(\varrho^A_n)^{-1}(\underline{a},x) = (\beta(x)^{-1} \cdot \underline{a}, x)$. Note that $\varrho^A_n$ is measure-class preserving; it is also $G$-equivariant, since
\[
    \varrho^A_n(k\cdot \underline{a},kx)=(\beta(kx)k\cdot \underline{a},kx)=(\sigma(k,x)\beta(x)\cdot \underline{a},kx)= k\ast \varrho^A_n(\underline{a},x)\,,
\]
and hence induces an isomorphism 
\begin{equation}
\varrho_A^n:\Linf(A_{\sigma}^{n+1}\times X,\tau^{\otimes n+1}\otimes \mu)^G \to \Linf(A^{n+1}\times X,\tau^{\otimes n+1}\otimes \mu)^G.
\end{equation}
We then obtain an isomorphism
\begin{equation}
i^n_A \coloneqq j^n_A\circ (\varrho^n_A)^{-1}: \Linf(A^{n+1} \times X, \tau^{\otimes n+1}\otimes \mu)^G \to \Linf(A^{n+1} \times Y, \tau^{\otimes n+1}\otimes \nu)^{\cG}.
\end{equation}
Explicitly, if $[f] \in  \Linf(A^{n+1} \times X, \tau^{\otimes n+1}\otimes \mu)^G$, then $i^n_A([f])$ is represented by
\[
h(\underline{a}, y) =  f|_{A^{n+1} \times Y}(\beta(y)^{-1} \cdot\underline{a}, y)
\]
However, by our normalization \eqref{BetaNormalization} we have $\beta(y) = e$, hence actually $h(\underline{a}, y) =  f|_{A^{n+1} \times Y}(\underline{a}, y)$, and hence $i^n_A$ is given by the same formula as $j^n_A$!
\end{no}

\begin{no}\label{BasicIsoBanach}
The remainder of this section is devoted to the proof of Theorem \ref{BanachIso}. By Corollary \ref{CorEta1} the isomorphisms $\eta_n$ and $\xi_n$ from Corollary \ref{IncludeS} are measure class preserving, hence yield mutually inverse $(\cG \times G)$-equivariant isometric isomorphisms
\[
\eta^{n}:  \Linf(A^{n+1}\times G\times Y , \tau^{\otimes n+1} \otimes m_G \otimes \nu) \to \Linf(A^{n+1}\times \mathcal{G}^{(1)}\times_Y X, \tau^{\otimes n+1} \otimes {\mu}_\rho)
\]
and 
\[
\xi^{n}: \Linf(A^{n+1}\times \mathcal{G}^{(1)}\times_Y X, \tau^{\otimes n+1} \otimes {\mu}_\rho)  \to \Linf(A^{n+1}\times G\times Y , \tau^{\otimes n+1} \otimes m_G \otimes \nu),
\]
which on the level of representatives are induced by $f \mapsto f \circ \eta_n$ and $g \mapsto g \circ \xi_n$ respectively. In particular, $\xi_n$ restrict to an isometric isomorphism
\[
\xi^{n}: \Linf(A^{n+1}\times \mathcal{G}^{(1)}\times_Y X, \tau^{\otimes n+1} \otimes {\mu}_\rho)^{\cG \times G}  \to \Linf(A^{n+1}\times G\times Y , \tau^{\otimes n+1} \otimes m_G \otimes \nu)^{\cG \times G}.
\]
We will deduce Theorem \ref{BanachIso} by identifying the domain and codomain of this isomorphism.
\end{no}
 \begin{no}\label{Isou} Denote by  $u_n: A^{n+1} \times G \times Y \to A^n \times Y$ the canonical projection; this is measure class preserving and intertwines the $\cG$-actions induced from the respective diagonal $G$-actions. As a consequence, it induces an isometric isomorphism 
  \begin{gather*}
  u^n:\Linf(A^{n+1} \times Y,\tau^{\otimes n+1} \otimes \nu)^{\cG} \to \Linf(A^{n+1}\times G\times Y,\tau^{\otimes n+1} \otimes m_G\otimes \nu)^{\cG \times G}
\end{gather*}
onto the domain of $\xi^n$. By definition, we have $u^n([f]) = [h]$ if and only if $h(\underline{a}, g, y) = f(\underline{a}, y)$ holds for $(\tau^{\otimes n+1} \otimes m_G\otimes \nu)$-almost all $(\underline{a}, g, y)$.
\end{no}
\begin{no}\label{RHS1} We now consider the codomain of $\eta^n$. For this we make the following elementary observations concerning the projection $\pi:  \mathcal{G}^{(1)}\times_Y X \to X$. Firstly, the projection is onto, since $[e,x] \in  \mathcal{G}^{(1)}\times_Y X$ for every $x \in X$ by \eqref{FiberProdRed}. Secondly, the action of $\cG$ on $\cG^{(1)} \times_Y X$ preserves the fibers of $\pi$. Finally, the $\cG$-action is transitive on fibers.  Indeed, if $[h_1, x], [h_2, x] \in \mathcal{G}^{(1)}\times_Y X$, then both $y_1 \coloneqq  h_1\beta(x) x$ and $y_2 \coloneqq  h_2\beta(x)  x$ lie in $Y$. If we set $g \coloneqq  h_2h_1^{-1}$, then
$(g, y_1) \in \cG$ with $(g, y_1)\cdot[h_1, x] = [h_2, x]$. To summarize, $\pi$ induces a bijection $\overline{\pi}: \cG \backslash (\cG^{(1)} \times_Y X) \to X$.

We further observe that the map $\mathrm{Id}_{A^{n+1}} \times \pi: A^{n+1} \times \mathcal{G}^{(1)}\times_Y X \to A_{\sigma}^{n+1} \times X$ is $G$-equivariant. It is also measure class preserving by Lemma \ref{PushRhoAway}, hence induces an isometric isomorphism
\[
(\mathrm{Id}_{A^{n+1}} \times \pi)^*: \Linf(A^{n+1}_\sigma \times X, \tau^{\otimes n+1}\otimes \mu)^G \to \Linf(A^{n+1} \times \mathcal{G}^{(1)}\times_Y X,  \tau^{\otimes n+1}\otimes \mu_{\rho})^{G \times \cG}.
\]
\end{no}
\begin{proof}[Proof of Theorem \ref{BanachIso}] The isometric isomorphisms $\xi^n$, $u^n$ and $(\mathrm{Id}_{A^{n+1}} \times \pi)^*$ from \S\ref{BasicIsoBanach}, \S\ref{Isou} and \S\ref{RHS1} determine an isometric isomorphism $j^n_A$ making the following diagram commute:
\[\begin{xy}\xymatrix{
\Linf(A^{n+1} \times G \times Y)^{\cG \times G} &&& \ar[lll]_{\xi^n} \Linf(A^{n+1} \times \cG^{(1)} \times_Y X)^{\cG \times_Y G}\\
\Linf(A^{n+1} \times Y)^{\cG} \ar[u]^{u^n} &&& \Linf(A_{\sigma}^{n+1} \times X)^G \ar[lll]_{j^n_A} \ar[u]_{(\mathrm{Id}_{A^{n+1}} \times \pi)^*}.
}\end{xy}\]
It remains to show that $j_A^n$ is given by restriction on the level of representatives. Let $\alpha \in \Linf(A_{\sigma}^{n+1} \times X)^G$ be represented by $f: A_{\sigma}^{n+1} \times X \to \bR$. Then $\xi^n((\mathrm{Id}_{A^{n+1}} \times \pi)^*\alpha)$ is represented by the Borel function $h$ given by
\[
h(\underline{a}, \gamma, y) = f((\mathrm{Id}_{A^{n+1}} \times \pi)(\beta(\gamma y)\gamma \cdot \underline{a},[\gamma^{-1}\beta(\gamma y)^{-1},\gamma y]))= f(\beta(\gamma y)\gamma \cdot \underline{a}, \gamma y).
\]
Since $\alpha$ is $G$-invariant, the function $h$ agrees almost everywhere with the function $h'$ given by $h'(\underline{a}, \gamma, y) \coloneqq h(\underline{a}, e, y)$; now
\[
h'(\underline{a}, \gamma, y) = f( \underline{a}, y) = f|_{A^{n+1} \times Y}(u_n(\underline{a}, \gamma, y)) =  u_n^*(f|_{A^{n+1} \times Y})(\underline{a}, \gamma, y).
\]
Since $u^n$ is an isomorphism, this shows that $j^n_A(\alpha)$ is represented by $f|_{A^{n+1} \times Y}$.
\end{proof}

\section{Bounded cohomology via amenable resolutions}\label{sec:amenable}

\subsection{Amenability}\label{sec:6.1}
\begin{no} By a \emph{Banach space with unit} we mean a Banach space $E$  with a distinguished element $1$. A \emph{unital morphism} of Banach spaces with units is a bounded linear map which maps the unit to the unit; if we consider $\bR$ as a Banach space with Euclidean norm and unit $1$, then a unital morphism $E \to \bR$ is called a \emph{mean} on $E$. If $(A, \tau)$ is a Lebesgue space, then we will always consider $\Linf(A,\tau)$ as a Banach space with unit $1_A$; we then recover the classical definition of a mean on $\Linf(A, \tau)$. 

If $\cE$ and $\cF$ are measurable bundles of Banach spaces with unit over a Lebesgue space $(\Omega,\tau)$, then a morphism $\cE \to \cF$ is \emph{unital} if it maps units to units in each fiber. A unital morphism $\L m: \cE \to \underline{\bR}$ is called a \emph{measurable system of means} for $\cE$. Note that we can pullback such systems along unital morphisms.

If $(\cG, \nu, \lambda)$ is a measured groupoid and $\cE$ and $\cF$ are measurable $(\cG, \nu, \lambda)$-bundles with units, then we can define unital $\cG$-morphisms $\cE \to \cF$ in the obvious way. We refer to a unital $\cG$-morphism $\L m: \cE \to \bR$ as an \emph{invariant system of means} for $\cE$.
\end{no}
From now on $(\cG, \nu, \lambda)$ denotes a measured groupoid. We write $\cL_\cG \coloneqq \cL^1$ for its first tautological bundle as in Example \ref{BEBF}.
\begin{defn}\label{DefAmenable}
The measured groupoid $(\cG, \nu, \lambda)$ is \emph{amenable} if its first tautological bundle $\cL_{\cG}$ admits an invariant system of means. A  Lebesgue $(\mathcal{G},\lambda)$-space $(A,\tau)$ is \emph{amenable} if the right action groupoid $(A \rtimes \cG, \tau, \lambda_A)$ is amenable.
\end{defn}
\begin{remark}
Definition \ref{DefAmenable} is one of many equivalent definitions of an amenable groupoid discussed in the classical book \cite{delaroche:renault:libro} by Anantharaman-Delaroche and Renault. To make the definition more explicit we recall from \eqref{fy} the isomorphism
\[
\Linf(\mathcal{G},\nu\circ \lambda)\to \Linf(\mathcal{G}^{(0)},\mathcal{L}_{\mathcal{G}})\,,\quad F \mapsto (F_y)_{y \in \cG^{(0)}},
\]
Then a system of means $\{\L m^y:\Linf(\mathcal{G}^y,\lambda^y)\to \mathbb{R}\}_{y\in \mathcal{G}^{(0)}}$ for $\cL_G$ is \emph{measurable} if for every $F \in \Linf(\mathcal{G},\nu\circ \lambda)$ the function $y \mapsto \L m^y(F_y)$ is $\nu$-measurable; it is \emph{invariant} if for every such $F$ and every $g\in \mathcal{G}^{(1)}$ we have
  \begin{equation}\label{eq:ISM}
    \L m^{s(g)} (g^{-1}\cdot F_{t(g)}) = \L m^{t(g)}(F_{t(g)}).
 \end{equation}
\end{remark}
\begin{example}\label{ExAmenable} We now spell out Definition \ref{DefAmenable} in various situations of interest.

\item (i) If $G$ is a lcsc group, then $\underline{G}$ is amenable iff there exists a $G$-invariant mean $\L m: \Linf(G,m_G) \to \bR$; this recovers the usual definition for lcsc groups.

\item (ii) Let $G$ be a lcsc group and let $(A, \tau)$ be a Lebesgue $G$-space. Given a function class $[F] \in  \Linf(A \times G,m_G\otimes \tau)$ we set $F_a(g) \coloneqq  F(a,g)$ so that $[F_a] \in \Linf(G,m_G)$. Then $(A, \tau)$ is amenable if and only if there exists a system of means $\{\L m^a:\Linf(G,m_G)\to \mathbb{R}\}_{a\in A}$ such that for every $F \in \Linf(A \times G,m_G\otimes \tau)$ the map $a\mapsto \L m^a(F_a)$ is measurable and 
\[
\L m^{g^{-1}\cdot a}(g^{-1}\cdot F_a)=\L m^a(F_a) \; \text{for $\tau$-almost every $a\in A$ and every $g\in G$.}
\]
It turns out that this condition is equivalent to amenability of the $G$-action on $(A, \tau)$ in the sense of Zimmer \cite{delaroche:renault:libro}.

\item (iii)  More generally, let $(\cG, \nu, \lambda)$ be a measured groupoid and let $(A,\tau)$ be a Lebesgue $(\mathcal{G},\lambda)$-space. We denote $\cR \coloneqq A \rtimes \cG$ and observe that $F \in \Linf(\cR^{(1)}, \tau \circ \lambda_A)$ corresponds to a family $(F_a)_{a \in A}$ with $F_a \in \Linf(\cG^{t_A(a)}, \lambda^{a})$ with $F_a(y) = F(a,g)$. Thus an invariant system of means is given by a family  $\{\L m^{a}:\Linf(\cG^{t_A(a)}, \lambda^{a})\to \mathbb{R}\}_{a\in A}$ such that for every $F \in \Linf(\cR^{(1)}, \tau \circ \lambda_A)$ the map $a\mapsto  \L m^a(F_a)$ is $\tau$-measurable and for $\tau$-almost every $a\in A$ we have
  \begin{equation}\label{eq:ISM:action}
    \L m^{g^{-1}\cdot a} (g^{-1}\cdot F_{ a}) = \L m^{ a}(F_a) \; \text{ for every $g\in\mathcal{G}$.}
   \end{equation}

\item (iv) We now specialize to our main case of interest: $G$ is a lcsc unimodular group, $(X, \mu, Y)$ is an integrable transverse $G$-system and $(\cG, \nu)$ is the associated transverse system. Now let $(A, \tau)$ be a Lebesgue $G$-space; we consider the induced $\cG$-space $(A \times Y, \tau \otimes \nu)$ and the corresponding right action groupoid $\cR \coloneqq (A \times Y) \rtimes \cG$. As in (ii) a bounded measurable functions on $\cR^{(1)}$ corresponds to a family $(F_{a,y})_{a,y \in A \times Y}$ with $F_{a,y} \in \Linf(\cG^y, \lambda^y)$ such that $F_{a,y}(g) = F(a,y,g)$, and an invariant system of means for $(A \times Y, \tau \otimes \nu)$ is given by a system of means $\{\L m^{a,y}:\Linf(\mathcal{G}^{y},\lambda^y)\to \mathbb{R}\}_{(a,y)\in A\times Y}$ such that for every $F \in \Linf(\cR^{(1)}, \tau \circ \lambda_{A \times Y})$
the map $(a,y) \mapsto  \L m^{a,y}(F_{a,y})$ is $\tau \otimes \nu$-measurable and 
  \begin{equation}\label{eq:ISM:action:AxY}
    \L m^{g^{-1}\cdot a,g^{-1}y} ((g, y)^{-1}\cdot F_{a,y}) = \L m^{ a,y}(F_{a,y}) 
  \end{equation}
  holds for $\tau\otimes \nu$-almost every $(a,y)\in A\times Y$ and every $g\in\mathcal{G}$.
\end{example}
We can now state the main result of this section:
\begin{thm}\label{InducedAmenable} If $(X, \mu, Y)$ is an integrable transverse system over a unimodular lcsc group $G$ with transverse groupoid $(\cG, \nu)$, then for every amenable $G$-space
$(A,\tau)$ the induced $\cG$-space  $(A\times Y,\tau\otimes \nu)$ is also amenable.
\end{thm}
\begin{proof} Assume that $(A, \tau)$ is an amenable $G$-space and fix an invariant system of means $\{\L m^a:\Linf(G,m_G)\to \mathbb{R}\}_{a\in A}$ as in Example \ref{ExAmenable}(ii).
Our goal is to construct an invariant system of measures for the tautological bundle $\cL_\cR$ over the right action groupoid $\mathcal{R}\coloneqq(A\times Y)\rtimes \mathcal{G}$.

We first construct an auxiliary bundle over $\cR$. Observe that, similar to Construction \ref{PullbackModules}, we can turn $\Linf(G)$ into a Banach-$\cR$-module via
\[
(a, y, g)\cdot f(h)\coloneqq g\cdot f(h)=f(g^{-1} h) \quad ((a,y, g) \in \cR, h \in G).
\]
We then denote by $\cE_G$ the associated constant bundle. We will first construct an invariant system of means for $\cE_G$ and then pull it back to an invariant system for $\cL_{\cR}$ via an equivariant bundle map $\cL_\cR \to \cE_G$.

If for every $(a,y) \in A \times Y$ we set $\L m^{a,y} \coloneqq \L m^a$, then for every $F \in  \Linf(A\times Y\times G)$ the map $(a,y)\mapsto \L m^{a,y}(F_{a,y})$ is $\tau\otimes \nu$-measurable; we thus obtain a measurable family $\L m = \{\L m^{a,y}:\Linf(G,m_G)\to \mathbb{R}\}_{(a,y)\in A\times Y}$ of means for the bundle $\cE_G$. Moreover, for every $F \in  \Linf(A\times Y\times G)$ and all $(a,y,g) \in \cR$ we have
\[
 \L m^{g^{-1}\cdot a,g^{-1}y} (g^{-1}\cdot F_{a,y})  =  \L m^{g^{-1}\cdot a} (g^{-1}\cdot F_{a,y}) =  \L m^a(F_{a,y})=\L m^{ a,y} (F_{a,y}),
\]
hence $\L m$ is invariant. 

We now construct an equivariant bundle map $K = (K_{a,y})_{(a,y) \in A \times Y}: \cL_{\cR} \to \cE_G$ as follows.  We fix a Borel map $\beta: X \to G$ and cocycle $\sigma: G\times X\to \Lambda(Y)$ as in Construction \ref{ConCocycle} and for every $y\in Y$ we define
\[\kappa_y: G\to \mathcal{G}\quad \kappa(g)=(\sigma(g^{-1},y),y)^{-1}\,.\] 
Since $\lambda^y$ is counting measure on $\cG^y$ we obtain a well-defined map
$$\kappa_y^*: \ell^\infty(\cG^y) = \Linf(\mathcal{G}^y,\lambda^y)\to \Linf(G,m_G)$$
 defined on representatives as $ \kappa_y^*(f)=f\circ \kappa_y$. Since $\sigma$ is measurable we obtain a measurable bundle map by setting $K_{a,y} \coloneqq \kappa_y^*$ for all $(a,y) \in A \times Y$, and it remains to show equivariance. 
 For all $f \in \ell^\infty(\cG^y)$, $(a,y,g) \in \cR$ and $\gamma \in \cR^{g \cdot a, gy} = G$ that
 \[
  K_{g\cdot a,gy}((g,y)\cdot f)(\gamma)  = (g,y) \cdot f(\kappa_{gy}(\gamma)) =  f((g,y)^{-1}(\sigma(\gamma^{-1},g y),gy)^{-1})
 \]
 Now observe that for all $(g,y) \in \cG$ we have
 \[
  \sigma(g,y)=\beta(gy)g\beta(y)^{-1} = g, \quad \text{since }\beta|_Y \equiv e.
 \]
 Combining this with the cocycle identity \eqref{CocycleId} we can rewrite the argument as
\begin{align*}
(g,y)^{-1}(\sigma(\gamma^{-1},g y),gy)^{-1} &= (\sigma(g,y), y)^{-1}(\sigma(\gamma^{-1},g y),gy)^{-1}\\
&= (\sigma(\gamma^{-1},g y)\sigma(g,y),y)^{-1} = (\sigma(\gamma^{-1}g,y),y)^{-1},
\end{align*}
hence we obtain
\[
K_{g\cdot a,gy}((g,y))\cdot f)(\gamma) = f((\sigma(\gamma^{-1}g,y),y)^{-1})= K_{a,y}(f) (g^{-1} \gamma ) = g \cdot K_{a,y} (f) (\gamma).
\]
This shows that $K:  \cL_{\cR} \to \cE_G$ is equivariant, and hence $K^*\L m$ is an invariant system of means for $\cL_{\cR}$.
 \end{proof}

\subsection{Measurable bounded cohomology of transverse measured groupoids}\label{SubsecAmenable}
Throughout this section, $(X, \mu, Y)$ denote an integrable transverse system over a unimodular lcsc group $G$ with transverse groupoid $(\cG, \nu)$.

\begin{con}[Induced amenable resolutions]\label{SRes}\label{ConAm} Let $(A, \tau)$ be an amenable Lebesgue-$G$-space. Then also  $(A^n, \tau^{\otimes n})$ is an amenable Lebesgue $G$-space (\cite[Example 5.4.1.(ii)]{monod:libro}), and hence the induced space $(A^n \times Y, \tau^{\otimes n} \otimes \nu)$ is an amenable $\cG$-space by Theorem \ref{InducedAmenable}. It then follows from \cite[Proposition~4.8]{sarti:savini25} that the associated $\mathcal{G}$-bundles $\mathcal{L}_{A^n\times Y}$ as defined in Construction \ref{LBundle} are relatively injective and as in \cite[Proposition~4.5]{sarti:savini25} we obtain a strong augmented resolution of $\underline{\bR}$ by setting
\begin{equation}\label{ABundles}
0 \to \underline{\bR} \xrightarrow{d^{-1}_A} \cL_{A \times Y} \xrightarrow{d^0_A} \cL_{A^2 \times Y} \xrightarrow{d^1_A} \cL_{A^3 \times Y} \xrightarrow{d^2_A} \dots,
\end{equation}
where fiberwise $d^{-1}_A$ is given by the inclusion of constants and $d^n_A$ is given by the alternating sum of the dual face maps of the simplicial space $A^{\bullet+1}$. It then follows from \S\ref{FunctorialCharacterization} that this resolution can be used to compute the measurable bounded cohomology of $(\cG, \nu)$. We thus obtain an isometric isomorphism
\begin{equation*}
(\cG \to A)^\bullet: \Hmb^\bullet((\cG, \nu); \underline{\bR}) \cong \Hm^\bullet\left(0 \to \Linf(A \times Y)^{\cG} \xrightarrow{d^0_A} \Linf(A^2 \times Y)^{\cG} \xrightarrow{d^1_A} \dots  \right)
\end{equation*}
which is induced by any equivariant chain map between the resolutions \eqref{StdRes} and \eqref{ABundles}.
\end{con}
\begin{con}[Functoriality]\label{Functoriality1}
We observe that if $(A, \tau)$ and $(B, \theta)$ are two amenable $G$-spaces, then we obtain canonical isomorphisms
\[
(A \to B)^k: \Hm^k(\Linf(A^{\bullet + 1} \times Y)^{\cG}, d_A^\bullet) \to \Hm^k(\Linf(B^{\bullet + 1} \times Y)^{\cG}, d_B^\bullet)
\]
such that the following diagram commutes:
\begin{equation}\label{CommutativeAmenable}
\begin{tikzcd}
 &&&  \Hm^k(\Linf(A^{\bullet+1} \times Y)^{\cG}) \arrow{dd}{(A \to B)^k}\\
\Hmb^k((\cG, \nu);\underline{\bR})\arrow{rrru} {(\cG \to A)^k} \arrow{rrrd}[swap]{(\cG \to B)^k}\\
 &&& \Hm^k(\Linf(B^{\bullet+1} \times Y)^{\cG}).
\end{tikzcd}
\end{equation}
Indeed, these isomorphisms can again be implemented by any equivariant chain map (compatible with augmentation) between the corresponding augmented bundle resolutions.
\end{con}
\subsection{The isomorphism theorem}\label{sec:6.2}
We now turn to the proof of our main theorem, Theorem \ref{Induction} from the introduction, which extends Corollary \ref{MainThmDiscrete} to the non-discrete case (and removes the ergodicity hypothesis). As in the theorem, we denote by $(\mathcal{G}, \nu)$ the transverse measured groupoid associated with an integrable transverse system $(X, \mu, Y)$ over a unimodular lcsc group $G$. We also fix a Borel map $\beta: X \to G$ and cocycle $\sigma: G\times X\to \Lambda(Y)$ as in Construction \ref{ConCocycle}. Our goal is to construct isometric isomorphisms
\begin{equation}
t^\bullet: \Hcb^\bullet(G; \Linf (X, \mu)) \to \Hmb^\bullet((\cG,\nu); \underline{\bR}).
\end{equation}
We are going to construct such isomorphisms using induced amenable resolutions as in Construction \ref{ConAm}. Moreover, it will be clear from our construction that these isomorphisms do not depend on the choice of amenable $G$-space which is used as an input and enjoy a certain form of functoriality. Our proof is modelled on the periodic case (Example \ref{LatticeCase}) and uses the same strategy as Monod in \cite[Lemma 5.4.3]{monod:libro}. 
\begin{no}\label{CbcFunct} By definition, the continuous bounded cohomology $\Hcb^k(G; \Linf(X, \mu))$ is defined as the cohomology of the complex
\[
0 \to \Linf(G \times X)^G \to \Linf(G^2 \times X)^G \to \dots
\] 
Alternatively, if $(A, \tau)$ is an amenable Lebesgue $G$-space, then it can be computed as the cohomology of the complex
\[
0 \to \Linf(A \times X)^G \to \Linf(A^2 \times X)^G \to \dots
\]
Moreover, this model of continuous bounded cohomology is natural in $(A, \tau)$ in the following sense: If $(A, \tau)$ and $(B, \theta)$ are amenable Lebesgue $G$-spaces, then there exists a $G$-equivariant chain map $i^\bullet_{A \to B}: \Linf(A^{\bullet+1} \times X) \to \Linf(B^{\bullet+1} \times X)$ which is compatible with augmentation, and hence induces isometric isomorphisms
\[
(A \to B)^k: \Hm^k(\Linf(A^{\bullet+1} \times X)^G) \to \Hm^k(\Linf(B^{\bullet+1} \times X)^G),
\]
which are independent of the choice of chain map $i^\bullet_{A \to B}$. Note that the chain map $i^\bullet_{A \to B}$ is not only equivariant with respect to the diagonal $G$-action, but also for the twisted $G$-action as introduced in \S\ref{TwistedAction}. It thus also induces isometric isomorphisms 
\[
(A_\sigma \to B_\sigma)^k: \Hm^k(\Linf(A_\sigma^{\bullet+1} \times X)^G) \to \Hm^k(\Linf(B_\sigma^{\bullet+1} \times X)^G)
\]
such that the following diagram commutes:
\begin{equation}\label{comm3}
  \begin{tikzcd}
  \Hm^k(\Linf(A^{\bullet+1}\times X)^{G})   \arrow{dd}[swap]{\Hm((\varrho_A^k)^{-1})} \arrow{rr}{(A\to B)^k}& &\Hm^k(\Linf(B^{\bullet+1} \times X)^{G}) \arrow[dd,"\Hm((\varrho_B^k)^{-1})"] \\
& & \\
\Hm^k(\Linf(A_{\sigma}^{\bullet+1}\times X)^{G}) \arrow{rr}{(A_{\sigma}\to B_{\sigma})^k}& &\Hm^k(\Linf(B_{\sigma}^{\bullet+1} \times X)^{G})"] \\
&&
\end{tikzcd}
\end{equation}
In particular, these isomorphisms are again independent of the choice of chain map. Everything we said so far applies in particular to the case $A=G$, and naturality means that we obtain a commutative triangle
\begin{equation}\label{CommAm2}
\begin{tikzcd}[scale cd=0.9,sep=small]
 &&&  \Hm^k(\Linf(A^{\bullet+1} \times X)^G) \arrow{dd}{(A \to B)^k}\\
 \Hcb^k(G;\Linf(X)) = \Hm^k(\Linf(G^{\bullet+1} \times X)^G)\arrow{rrru} {(G \to A)^k} \arrow{rrrd}[swap]{(G \to B)^k}\\
 &&& \Hm^k(\Linf(B^{\bullet+1} \times X)^G).
\end{tikzcd}
\end{equation}
\end{no}
\begin{con}[Induction isomorphism]\label{IndIso} If $A$ is an amenable $G$-space, then by Theorem \ref{BanachIso} we have an isometric transfer isomorphism $i^{k}_A: \Linf(A^{k+1} \times X)^G \to  \Linf(A^{k+1} \times Y)^{\cG}$ defined as a composition 
  $i^{k}_A= j^n_A\circ (\varrho^k_A)^{-1}$ where $j^n_A$ is
  given by restriction of representatives. On the other hand,  by Construction \ref{ConAm} we have an isometric isomorphism $(A \to \cG)^k: \Hm^k( \Linf(A^{\bullet+1} \times Y)^{\cG}) \to \Hmb^k((\cG, \nu);\underline{\bR})$. If $(B, \tau)$ is another amenable $G$-space, then by Construction \ref{ConAm} we have an isomorphism \[(A \to B)^k: \Hm^k( \Linf(A^{\bullet+1} \times Y)^{\cG}) \to \Hm^k( \Linf(B^{\bullet+1} \times Y)^{\cG})\,.\]
Combining the commutative diagrams \eqref{CommutativeAmenable}, \eqref{comm3} and \eqref{CommAm2} and using the explicit form of the transfer isomorphisms, we obtain for every $k \geq 0$ the following commutative diagram of isometric isomorphisms:
\begin{equation}\label{MainDiagram}
  \begin{tikzcd}[scale cd=0.9,sep=small]
    &  \Hcb^k(G;\Linf(X)) \arrow{ddl}[swap]{(G \to A)^k} \arrow{ddr}{(G \to B)^k}& \\
    &&\\
\Hm^k(\Linf(A^{\bullet+1}\times X)^{G})   \arrow{dd}[swap]{\Hm((\varrho_A^k)^{-1})} \arrow{rr}{(A\to B)^k}& &\Hm^k(\Linf(B^{\bullet+1} \times X)^{G}) \arrow[dd,"\Hm((\varrho_B^k)^{-1})"] \\
& & \\
\Hm^k(\Linf(A_{\sigma}^{\bullet+1}\times X)^{G})   \arrow{dd}[swap]{\Hm(j_A^k)} \arrow{rr}{(A_{\sigma}\to B_{\sigma})^k}& &\Hm^k(\Linf(B_{\sigma}^{\bullet+1} \times X)^{G}) \arrow[dd,"\Hm(j_B^k)"] \\
&&\\
\Hm^k(\Linf(A^{\bullet+1}\times Y)^{\mathcal{G}})   \arrow{rdd}[swap]{(A\to \mathcal{G})^k} \arrow{rr}{(A\to B)^k}& &\Hm^k(\Linf(B^{\bullet+1} \times Y)^{\mathcal{G}}) \arrow[ldd,"{(B\to \mathcal{G})^k}"] \\
&&\\
&\Hmb^k((\mathcal{G},\nu);\underline{\mathbb{R}})&
  \end{tikzcd}
\end{equation}

Concatenating these isometric isomorphisms then yields the desired isometric isomorphism \[t^k: \Hcb^k(G;\Linf(X,\mu)) \to  \Hmb^k((\mathcal{G},\nu);\underline{\mathbb{R}}),\] and commutativity of the diagram shows that the isomorphism
is independent of the choice of amenable $G$-space $(A, \tau)$ which was  used to define it.
\end{con}
This concludes the proof of Theorem \ref{Induction}. 
 \begin{remark}
 One would like to also adapt the proofs of \cite[Proposition 10.1.3 and 10.1.5]{monod:libro} to the present setting to obtain a stronger form of functoriality. This, however, would require a functorial characterization of measurable bounded cohomology for the non-discrete measured groupoid $\mathcal{G}\times G$. For the moment, such a characterization is only available in the case of discrete measured groupoids, hence the above functoriality is the best we can get at this point.
 \end{remark}
 \subsection{An explicit implementation via harmonic cocycles}\label{SecPoisson}
 We keep the notation of the previous section; our goal is to provide an explicit implementation of the induction isomorphism $t^n: \Hcb^n(G; \Linf(X)) \to \Hmb^n((\cG, \nu); \underline{\bR})$ on the cocycle level. For this we are going to need some standard results about Poisson boundaries; a convenient reference is \cite{Furman_RandomWalksOnGroups}.
\begin{no}\label{Poisson1} A probability measure $p$ on $G$ is called \emph{admissible} if some convolution power $p^{\ast n}$ is absolutely continuous with respect to $m_G$ and the support of $p$ generates $G$ as a semigroup. For the remainder of this section we fix an admissible probability measure $p$ and denote by $\cH^\infty(G) = \cH^\infty_p(G)$ the space of bounded measurable functions $f: G \to \bR$ which are \emph{$p$-harmonic} in the sense that 
\[
f(g) = \int_G f(gg')\, dp(g') \quad \text{for all }g\in G.
\]
Any such function is automatically continuous by admissibility of $p$, and there is an amenable Lebesgue $G$-space $(B, \tau) = (B_p, \tau_p)$, called the \emph{Furstenberg-Poisson boundary} of $(G, p)$, such that the \emph{Poisson transform}
\[
\cP: \Linf(B, \tau) \to \cH^\infty(G), \quad \cP f(g) = \int_{B} f(g\xi) \, d\tau(\xi)
\]
is an isometric isomorphism. The construction is compatible with products, i.e. for every $n \in \bN$ we obtain an isometric isomorphisms $\cP^n: \Linf(B^{n}, \tau^{\otimes n}) \to \cH^\infty(G^n)$ by coordinate-wise action and integration. This induces a chain map \[\cP^{\bullet+1}: \Linf(B^{\bullet + 1} \times X, \tau^{\otimes \bullet+1} \otimes \mu)^G \to \Linf(G^{\bullet + 1}\times  X, m_G^{\otimes n+1}\otimes \mu)^G\]
which on the level of cohomology induces the isomorphism $(G \to B)^{\bullet}$ from \S\ref{CbcFunct}. In particular, every bounded cohomology class with coefficient in $L^\infty(X)$ can be represented by a Borel function which is harmonic (and hence continuous) in the $G$-variables.
\end{no}
\begin{con}[Poisson transform as a chain map]\label{Poisson2} As in \S\ref{Poisson1} we fix an admissible probability measure $p$ on $G$ and denote by $(B, \tau)$ the associated Poisson boundary. In order to implement the canonical maps $(B \to \cG)^\bullet$ from Construction \ref{Functoriality1}, we need to construct a family $P^n =  (P^n_y: \Linf(B^{n} \times \{y\},\tau^{\otimes n}\otimes \delta_y ) \to \ell^\infty((\cG^y)^n))_{y \in \cG^{(0)}}$ of $\cG$-bundle morphisms
such that the following diagram commutes:
\[\begin{xy}\xymatrix{
0 \ar[r] & \underline{\bR} \ar[r] \ar[d]^{\mathrm{Id}} & \cL_{B \times Y} \ar[d]^{P^1}\ar[r] & \cL_{B^2 \times Y}\ar[d]^{P^2}\ar[r] & \cL_{B^3 \times Y}\ar[d]^{P^3} \ar[r] &\dots\\
0 \ar[r]& \underline{\bR} \ar[r]& \cL^0\ar[r] & \cL^1\ar[r] & \cL^2 \ar[r] &\dots
}\end{xy}\]
To construct such a chain map $P^\bullet$, we observe that the Poisson transform induces isomorphisms $\cP_y:  L^\infty(B^n \times \{y\},\tau^{\otimes n}\otimes \delta_y) \to \cH^\infty(G^n)$, where $\cP_y f$ is the Poisson transform of the function $(\xi_0, \dots, \xi_{n-1}) \mapsto f(\xi_0, \dots, \xi_{n-1}, y)$. We can compose this with the natural restriction map $\cH^\infty(G^n) \to \ell^\infty((\cG^y))^n$ to obtain the desired chain map. More explicitly, if  $((g_0, y_0), \dots, (g_{n-1}, y_{n-1})) \in (\cG^{(1)})^n$ with $g_0 y_0 = \dots = g_{n-1} y_{n-1} = y$, then $P^n_yf((g_0, y_0), \dots, (g_{n-1}, y_{n-1}))  = \cP f(g_0, \dots, g_{n-1})$.
\end{con}
\begin{prop}\label{ImplementationII} Let $\alpha \in\Hcb^n(G; \Linf(X))$ and $\alpha_{\cG} := t^n(\alpha) \in \Hmb^n((\cG, \nu); \underline{\bR})$. Then
\begin{enumerate}[(i)]
\item $\alpha$ can be represented by a bounded Borel function $f: G^{n+1} \times X \to \bR$ which is harmonic in the $G$-variables.
\item If $f$ is as in (i), then  $\alpha_{\cG}$ is represented by the function $f_{\cG}$ defined as follows: If $((g_0, y_0), \dots, (g_{n}, y_{n})) \in (\cG^{1})^{n+1}$ with $g_0y_0 = \dots = g_{n-1} y_{n-1} = y$, then 
\[f_{\cG}((g_0, y_0), \dots, (g_{n}, y_{n})) = f|_{G^{n+1} \times Y}(g_0, \dots, g_n, y).\]
\end{enumerate}
\end{prop}
\begin{proof} (i) This was established in \S\ref{Poisson1}.

\item (ii) We trace $\alpha$ through the diagram \eqref{MainDiagram} and denote by $\alpha', \alpha'', \alpha'''$ the images of $\alpha$ in $\Hm^n(\Linf(B^{\bullet+1} \times X)^{G})$, $\Hm^n(\Linf(B_\sigma^{\bullet+1} \times X)^{G})$ and $\Hm^n(\Linf(B^{\bullet+1} \times Y)^{\cG})$ respectively. If $f$ is as in (i), then by \S\ref{Poisson1} there exists a bounded Borel function $f^+: B^{n+1} \to \bR$ such that $\cP^n([f^+]) = f$, and we fix such a function once and for all; then $\alpha' = [f^+]$ and $\alpha'' = [f^+_\beta]$, where
\[
f^+_{\beta}: B^{n+1} \times X \to \bR, \quad (\xi_0, \dots, \xi_n, x) \mapsto   f^+(\beta(x)^{-1}\cdot \xi_0, \dots, \beta(x)^{-1}\cdot \xi_n, x)
\]
Consequently, $\alpha'''$ is represented by the restriction of $f^+_{\beta}$ to $B^{n+1} \times Y$, which by \eqref{BetaNormalization} is the same as the restriction of $f^+$ to $B^{n+1} \times Y$. Since 
\[\cP^{n+1}(f^+|_{B^{n+1} \times Y}) = (\cP^{n+1} f^+)|_{G^{n+1} \times Y} = f|_{G^{n+1} \times Y}\]
The proposition then follows from Construction \ref{Poisson2}.
\end{proof}

 \section{Towards applications}\label{sec:res:ind}

Throughout this section, $(X, \mu, Y)$ denote an integrable transverse system over a unimodular lcsc group $G$ with transverse groupoid $(\cG, \nu)$.

 \subsection{Restriction and induction}\label{sec:7.1}
 \begin{no}\label{PreResInd}
We denote by $i_c: \bR \to \Linf(X, \mu)$ the inclusion of constants and set
 \[
 \Linf(X, \mu)_0 \coloneqq  \left\{[f] \in \Linf(X, \mu) \mid \int f \, d \mu = 0 \right\}
 \]
 We then have a short exact sequence of coefficient $G$-modules given by
 \begin{equation}\label{SesInt}
 0 \to \bR \xrightarrow{i_c} \Linf(X, \mu) \xrightarrow{p_\mu} \Linf(X, \mu)_0 \to 0,
 \end{equation}
 where
 \[
 p_\mu(f) \coloneqq  f- \int f \, d\mu \cdot 1_X.
 \]
 The short exact sequence \eqref{SesInt} is actually split-exact, since we can define an explicit left-splitting by
 \[
 s_\mu:  \Linf(X, \mu) \to \bR, \quad f \mapsto \int_X f \, d \mu.
 \] 
 This implies that for every $k \geq 0$ the map $(i_c)^k_*:  \Hcb^k(G; \bR) \to  \Hcb^k(G; \Linf(X, \mu))$ is injective and its left-inverse $(s_\mu)^k_*: \Hcb^k(G; \Linf(X, \mu)) \to \Hcb^k(G; \bR)$ is surjective; moreover, we have
 a short exact sequence
 \[
 0 \to \Hcb^k(G; \bR) \xrightarrow{(i_c)^k_*} \Hcb^k(G; \Linf(X, \mu)) \xrightarrow{(p_\mu)^k_*} \Hcb^k(G; \Linf(X, \mu)_0) \to 0.
 \]
 Note that, by definition, the maps $i_c$ and $s_\mu$ are norm-non-increasing, and hence $(i_c)^k_*$ and $(s_\mu)^k_*$ are seminorm-non-increasing for every $k \geq 0$.
 \end{no}
 \begin{defn}\label{def:res} The \emph{restriction} and \emph{induction map} are defined as the linear maps
 \[
 \mathrm{res}^k_{G,\cG}: \Hcb^k(G; \bR) \to \Hmb^k((\cG, \nu); \underline{\bR}) \qand \mathrm{ind}^k_{\cG,G}: \Hmb^k((\cG, \nu); \underline{\bR}) \to  \Hcb^k(G; \bR)
 \]
 given by composing $(i_c)^k_*$ and $(s_\mu)^k_*$ respectively with the isomorphism $t^k$ from Construction \ref{IndIso}.
 \end{defn}
 From \S \ref{PreResInd} we deduce:
 \begin{prop}\label{prop:res:ind}
   For every $k \geq 0$ the map $\mathrm{res}^k_{G,\cG}$ is injective, its left-inverse $\mathrm{ind}^k_{\cG,G}$ is surjective, and both of these maps are seminorm-non-increasing.
 Moreover, there is a short exact sequence
 \[
 0 \to \Hcb^k(G; \bR) \xrightarrow{\mathrm{res}^k_{G, \cG}} \Hmb^k((\cG, \nu); \underline{\bR}) \to  \Hcb^k(G; \Linf(X, \mu)_0) \to 0,
 \]
 hence $\mathrm{res}^k_{G, \cG}$ is an isometric isomorphism if and only if $\Hcb^k(G; \Linf(X, \mu)_0) = 0$. 
 \end{prop}
 From Proposition \ref{ImplementationII} we deduce:
 \begin{cor}[Implementation of restriction]\label{ImplementRes} Let $c: G^{n+1} \to \bR$ be a bounded harmonic $G$-invariant cocycle representing a class $\alpha \in \Hcb^n(G; \bR)$. Then $\mathrm{res}(c)$ is represented by the cocycle $c_{\cG}$ given by
 \[
 \pushQED{\qed}
 c_{\cG}((g_0, y_0), \dots, (g_n, y_n)) = c(g_0, \dots, g_n).\qedhere \popQED
 \]
 \end{cor}

\subsection{Proof of corollaries} \label{sec:corollaries}
We conclude by proving all the corollaries stated in the introduction. From now on, $(\cG, \nu)$ denotes the transverse measured groupoid of an integrable transverse system $(X, \mu, Y)$ over a unimodular lcsc group $G$.
\begin{proof}[Proof of Corollary \ref{intro_cor_amenable}] If $G$ is amenable then it is boundedly acyclic for every coefficient $G$-module $E$ \cite[Corollary 7.5.11]{monod:libro}. Applying this for $E=\Linf(X,\mu)$ and using Theorem \ref{Induction}, we obtain
 \[\Hmb^k((\mathcal{G},\nu);\underline{\mathbb{R}})\cong\Hcb^k(G;\Linf(X,\mu))=0\,,\quad k>0\,.\qedhere\]
\end{proof}
\begin{proof}[Proof of Corollary \ref{intro_cor_seminorms}] This is contained in Proposition \ref{prop:res:ind}.
\end{proof}
\begin{proof}[Proof of Corollary \ref{2rk}] We recall that a coefficient $G$-module $E$ is \emph{semi-separable} if admits a dual injection into another coefficient $G$-module which is separable \cite{MonodSarithmetic}. Since $\Linf(X,\mu)$ embeds in $\textrm{L}^2(X,\mu)$ and the injection map is dual, $\Linf(X,\mu)$ is semi-separable and so is $\Linf(X,\mu)_0$. By our ergodicity assumption, the latter has no invariant vectors for the $G$-action, hence \cite[Theorem 1.2]{MonodSarithmetic} applies and yields 
\[\Hcb^k(G;\Linf(X,\mu)_0)=0\,,\quad k<2\rk_{\mathbb{R}}(G)\,.\]
The corollary the follows from Proposition \ref{prop:res:ind}.
\end{proof}
\begin{proof}[Proof of Corollary \ref{intro_cor_2BC}] 
This follows by combining Corollary \ref{2rk} with \cite{BuMo} (for degree $2$) and the main result of \cite{DLCM} (for degree $3$).
\end{proof}
\begin{proof}[Proof of Corollary \ref{intro_cor_Archimedean}]
 By \cite[Theorem A]{Flatmates} we have  $\Hcb^k(G;\mathbb{R})=0$ for all $k>0$, hence we can argue as in the proof of Corollary \ref{2rk} to conclude.
\end{proof}
\begin{proof}[Proof of Corollary \ref{ApproxLatt1}] This is immediate from Theorem \ref{Induction} and the definitions.
\end{proof}
\begin{proof}[Proof of Corollary \ref{ApproxLatt2}] This follows by combining Corollary \ref{intro_cor_2BC} and Corollary \ref{ApproxLatt1}.
\end{proof}
\begin{proof}[Proof of Corollary \ref{Discretization}] 
Assume that $(X,\mu, Y)$ is a transverse $G$-system with transverse measured groupoid $(\cG, \nu)$ such that $\Lambda(Y)$ is discrete. It then follows from Corollary \ref{ImplementRes} that the restriction map $\Hcb^{\bullet}(G; \bR) \to \Hmb^{\bullet}((\cG, \nu); \underline{\bR})$ factors through the discretization map $\Hcb^\bullet(G; \bR) \to \Hb^\bullet(G^\delta; \bR)$, and hence injectivity of the former implies injectivity of the latter.
\end{proof}

  \bibliographystyle{alpha}

\bibliography{biblionote}

\end{document}